\documentclass[a4paper]{amsart}
\usepackage[all]{xy}
\usepackage{enumerate}
\usepackage{amssymb}
\usepackage{color} 

\numberwithin{equation}{section}

\newtheorem{thm}{Theorem}[section]

\newtheorem{prop}[thm]{Proposition}
\newtheorem{lem}[thm]{Lemma}

\theoremstyle{definition}
\newtheorem{defn}[thm]{Definition}
\newtheorem{ex}[thm]{Example}

\theoremstyle{remark}

\newcommand{\Rr}{\mathbb R}
\renewcommand{\d}{\mathrm d}                              

\newcommand{\br}[1]{\left[#1\right]}                      
\newcommand{\pbr}[1]{\left\{#1\right\}}                   
\newcommand{\pbrEm}{\left\{\cdot,\cdot\right\}}           
\newcommand{\lie}{\mathcal{L}}                            
\newcommand{\pd}[2]{\frac{\partial #1}{\partial #2}}      

\newcommand{\X}{\ensuremath{\mathfrak{X}}}                
\newcommand{\Ha}{\ensuremath{\mathcal{H}}}                
\newcommand{\Ree}{\ensuremath{\mathcal{R}}}               
\newcommand{\g}{\ensuremath{\mathfrak{g}}}                

\DeclareMathOperator{\Ad}{Ad}                             
\DeclareMathOperator{\hor}{hor}                           

\begin{document}

\title{Reduction of symplectic principal $\Rr$-bundles}

\author{Ignazio Lacirasella}
\address{I.\ Lacirasella: 
Dipartimento di Matematica, Universit\`a degli Studi Aldo Moro di Bari, Bari, Italy}
\email{ignazio.lacirasella@uniba.it}

\author{Juan Carlos Marrero}
\address{J.\ C.\ Marrero:
Unidad asociada ULL-CSIC {\it Geometr{\'\i}a Diferencial y Mec\'anica Geom\'etrica.} Departamento de Matem\'atica Fundamental, Universidad de la Laguna, 38071 La Laguna, Tenerife
Canary Islands, Spain} \email{jcmarrer@ull.es}

\author{Edith Padr{\'o}n}
\address{E.\ Padr{\'o}n: 
Unidad asociada ULL-CSIC {\it Geometr{\'\i}a Diferencial y Mec\'anica Geom\'etrica.} Departamento de Matem\'atica Fundamental, Universidad de la Laguna, 38071 La Laguna, Tenerife
Canary Islands, Spain} \email{mepadron@ull.es}

\begin{abstract}
 {We describe a reduction process for symplectic principal $\Rr$-bundles in the presence of a momentum map. 
This type of structures plays an important role in the geometric formulation of non-autonomous Hamiltonian systems. 
We apply this procedure to the standard symplectic principal $\Rr$-bundle associated with a fibration $\pi:M\to\Rr$. 
Moreover, we show a reduction process for non-autonomous Hamiltonian systems on symplectic principal $\Rr$-bundles. We apply these reduction processes to several examples.}
\end{abstract}

\thanks{\noindent \emph{Acknowledgments:} The research has been 
partially supported by MEC (Spain) grants MTM2009-13383,
MTM2010-12116-E, and the Projects of the Canary government
SOLSUBC200801000238 and ProdID20100210. 
The authors  would like to thank the referees for their very valuable comments and suggestions which have  improved our article. 
L.~would like to thank the University of 
Bari and, in particular, Prof.~Pastore for the support.}

\maketitle



\section{Introduction}

It is well-known that the configuration space  for a non-autonomous 
mechanical system is a smooth manifold which is fibered on the real line. 
So, we have a fibration $\pi:M\to\Rr$, with respect to which the restricted 
phase space of momenta is the dual bundle $V^*\pi$ of the vertical bundle 
$V\pi$ of $\pi$ and the extended phase space of momenta is the cotangent 
bundle $T^*M$ of $M$.

In the presence of a Hamiltonian section (that is, a section of the 
canonical projection $\mu_{\pi}:T^*M\to V^*\pi$) and using the canonical 
symplectic structure of $T^*M$ (respectively, a suitable cosymplectic 
structure on $V^*\pi$)  one may develop the extended (respectively, the restricted) Hamiltonian formalism 
(see Section \ref{section2} and \cite{GrabGrabUrbAV, IgMaPaSo, LeMaMa}). 

We remark that $\mu_\pi:T^*M\to V^*\pi$ is  a principal $\Rr$-bundle  (an {\it AV-bundle } in the terminology of \cite{GrabGrabUrbAV}). In addition, the principal action is symplectic 
and, thus, we have a symplectic principal $\Rr$-bundle. 
A non-autonomous Hamiltonian system is a symplectic principal $\Rr$-bundle $\mu:A\to V$ 
and a Hamiltonian  section $h:V\to A$, that is, a section of the principal 
$\Rr$-bundle projection $\mu$. The Hamiltonian section $h$ induces a vector 
field on $V$ whose integral curves are the solutions of the dynamical equations 
for the Hamiltonian system (see Section \ref{section5.1}). In the particular case 
when $\mu:A\to V$ is a standard symplectic principal $\Rr$-bundle (that is, $\mu=\mu_{\pi}$ 
for some fibration $\pi:M\to\Rr$), we obtain the classical Hamilton equations. 
Moreover, all the tools used in the standard theory in geometric non-autonomous 
mechanics, as Lagrangian and Hamiltonian formalisms or variational formulation, 
appear in the framework of the AV-bundles 
(see \cite{GrabGrabUrbAV, GrGrUrbEL, GrUrVariational, GrUr}).

On the other hand, in the context of autonomous mechanical systems, the phase space 
is represented by a symplectic manifold. A classical procedure, due to Marsden and 
Weinstein (\cite{MarsWein}), called \emph{symplectic reduction}, allows 
to use symmetry pro\-per\-ties of a symplectic manifold in order 
to reduce the degrees of freedom of the system. Moreover, if an invariant Hamiltonian function 
is given on the unreduced symplectic manifold, one may obtain a Hamiltonian system on 
the reduced symplectic manifold.

Reduction theory may be also applied, 
for example, in order to obtain the symplectic structure on the coadjoint orbit of 
a Lie group $G$ (see \cite{AbrMarsd}). In the particular case when the unreduced 
symplectic manifold is the cotangent bundle of a manifold endowed with its 
canonical symplectic structure, a natural question arises: is the reduced symplectic 
manifold a standard symplectic manifold (that is, a cotangent bundle endowed with 
its canonical symplectic structure)? 
An answer to this question is given by the so-called {\it cotangent bundle reduction theory} 
(see \cite{Ku,MarsEtAl,MaMoRa,OrRa}). 
 
Although reduction theory goes back to the early roots of mechanics, 
it allows to obtain many results about other geometric structures. 
Indeed, a similar idea may be used in order to reduce not only Poisson structures (\cite{MarsRat}), 
but also cosymplectic, K\"ahler, hyperk\"ahler, contact, $f$-structures, etc.~
and to obtain new examples of such a kind of manifolds (\cite{Alb,DiTKo,Ge,GuSt,HiKaLiRo}).

{The aim of this paper is to perform the reduction process in 
the framework of symplectic principal $\Rr$-bundles. We introduce the notion of 
symmetry of a symplectic principal $\Rr$-bundle and show that, 
under suitable regularity conditions, one 
may obtain a reduced symplectic principal $\Rr$-bundle. We apply this reduction process to the standard symplectic principal $\Rr$-bundle associated with a fibration.  Finally, we prove that a non-autonomous Hamiltonian system with equivariant Hamiltonian section induces a non-autonomous Hamiltonian system on the reduced symplectic principal $\Rr$-bundle.

The paper is structured as follows. In Section 2, we recall some basic facts about non-autonomous Hamiltonian 
systems which motivate the study of symplectic AV-differential geometry. In Section 3 we introduce the category of 
symplectic principal $\Rr$-bundles and prove that the base manifold $V$ of a symplectic 
principal $\Rr$-bundle is canonically equipped with a Poisson structure. Moreover, we will relate the induced Poisson structures on the corresponding base spaces of an embedding of symplectic principal $\Rr$-bundles. 
In order to introduce a reduction process for symplectic principal $\Rr$-bundles, in Section 4, we define the notion of a \emph{canonical action} on a principal $\Rr$-bundle and prove 
the reduction theorem in the symplectic principal $\Rr$-bundle framework. As an example, in the last part of this section, we discuss the reduction of a standard symplectic principal $\Rr$-bundle $\mu_{\pi}:T^*M\to V^*\pi$ associated with a fibration $\pi:M\to\Rr$ which is invariant with respect to a free and proper action of a Lie group $G$ on $M$. In Section 5, we develop the reduction of a non-autonomous Hamiltonian system on a symplectic principal $\Rr$-bundle. For this purpose, we use symplectic and cosymplectic reduction theory. 
In Sections 4 and 5 we apply the reduction processes to the case of the bidimensional time-dependent damped harmonic oscillator and the time-dependent heavy top. Finally, in Section \ref{Section6}, we show how apply these reductions  to the frame-independent formulation of the analytical mechanics in the Newtonian space-time. 

The paper ends with our conclusions, a description of future research directions and an Appendix in which we review some reduction processes for  Poisson, symplectic and cosymplectic manifolds. }

\section{A motivation: non-autonomous Hamiltonian systems}\label{section2}

It is well-known that if a manifold $Q$ is the configuration space of an autonomous Hamiltonian system, 
then $T^*Q$ is the phase space of momenta. Moreover, using the canonical symplectic structure of $T^*Q$, 
one may describe the Hamilton equations in an intrinsic form (see, for instance, \cite{AbrMarsd}). 

For non-autonomous Hamiltonian systems the situation is different (see, for instance, 
\cite{GrabGrabUrbAV,IgMaPaSo,LeMaMa}). Namely, \emph{the configuration space} is a manifold $M$ fibered 
over the real line. So, we have a surjective submersion $\pi:M\to\Rr$. We will denote by $V\pi$ the 
vertical bundle of $\pi$ which is a vector bundle over $M$. Then, \emph{the restricted} (respectively, 
\emph{extended}) \emph{phase space of momenta} is the dual bundle $V^*\pi$ of $V\pi$ 
(respectively, the cotangent bundle $T^*M$ of $M$). 

 We remark that $T^*M$ is a principal $\Rr$-bundle over $V^*\pi$ and the dual map $\mu_\pi:T^*M\to V^*\pi$ of the inclusion $i:V\pi\to TM$
is the principal bundle projection. 
The corresponding principal action $\psi_\pi:\Rr\times T^*M\to T^*M$ is defined by
\[
 \psi_\pi(s,\alpha_x)=\alpha_x+s\pi^*(\d t)(x), \qquad\text{for $s\in\Rr$ and $\alpha_x\in T^*_xM,$}
\]
where $t$ is the usual coordinate on $\Rr$. 
 Note that the principal action $\psi_\pi$ is symplectic with  respect to the symplectic structure on $T^*M.$ Moreover,  the infinitesimal generator $Z_{\mu_\pi}$ of $\psi_\pi$ 
is the Hamiltonian vector field $$Z_{\mu_{\pi}}={\mathcal H}_{ -\pi\circ\pi_M}\in {\mathfrak X}(T^*M)$$ of the real function on $T^*M$ given by 
$-\pi\circ\pi_M:T^*M\to\Rr,$
where $\pi_M:T^*M\to M$ is the canonical projection. One proves easily that
\begin{equation}\label{muprojectable}
\mbox{$f\in C^\infty(T^*M)$ is $\mu_{\pi}$-projectable if and only if $Z_{\mu_{\pi}}(f)=0.$}
\end{equation}

On the other hand, the extended phase of momenta $T^*M$ admits a linear Poisson structure $\pbrEm_{T^*M}$ induced by the canonical symplectic $2$-form $\Omega_M$ on $T^*M$. Moreover, using (\ref{muprojectable}) and the Jacobi identity of $\{\cdot,\cdot\}_{T^*M},$ one deduces  that the subset 
 $\mu_\pi^*(C^\infty(V^*\pi))$ of $C^\infty(T^*M)$ is closed with respect to $\pbrEm_{T^*M}$. Therefore, there is a unique Poisson structure $\pbrEm_{V^*\pi}$ on $V^*\pi$ such that  (see \cite{Vais})
 \begin{equation}\label{linearPoissonVpi}
 \pbr{f\circ\mu_\pi,h\circ\mu_\pi}_{T^*M}=\pbr{f,h}_{V^*\pi}\circ\mu_\pi,\qquad f,h\in C^\infty(V^*\pi).
\end{equation}
Notice that $\pbrEm_{V^*\pi}$ is also linear and  $\mu_\pi:T^*M\to V^*\pi$ is, by construction, a Poisson epimorphism.

In this setting, \emph{a Hamiltonian section} is a section $h:V^*\pi\to T^*M$ of $\mu_\pi$. Using the Hamiltonian section 
one may define a cosymplectic structure $(\omega_h,\eta)$ on $V^*\pi$ as follows 
\begin{equation}\label{OmhAndEtaTilde}
 \omega_h=h^*(\Omega_M), \qquad \eta=\pi_{V^*\pi}^*(\pi^*(\d t))
\end{equation}
with $\pi_{V^*\pi}:V^*\pi\to M$ the corresponding projection (for the definition of a cosymplectic structure, see Appendix \ref{MWRed}). In fact, the $1$-form $\eta$ given by \eqref{OmhAndEtaTilde} is just  
$\eta=-h^*(i_{Z_{\mu_\pi}}\Omega_M)$ and $\omega_h$ is well known  Poincar\'e-Cartan $2$-form.

On the other hand, since 
 $\mu_\pi((h\circ\mu_\pi)(\alpha_x))=\mu_\pi(\alpha_x),$
there exists a unique $F_h(\alpha_x)\in\Rr$ such that 
\[
 \psi_\pi(-F_h(\alpha_x),\alpha_x)=h(\mu_\pi(\alpha_x)).
\]
The \emph{extended Hamiltonian function} associated with the Hamiltonian section $h$ is just the real $C^\infty$-function 
$F_h:T^*M\to\Rr$ and $(T^*M,\Omega,F_h)$ is the so-called  \emph{homogeneous Hamiltonian system}. It's easy to prove that the Hamiltonian vector field $\Ha_{F_h}$ of $F_h$ is $\mu_\pi$-projectable on the 
Reeb vector field $R_h$ of the cosymplectic structure $(\omega_h,\eta)$.

In what follows, we will give the local expressions of these elements. Firstly, from the fact that  $\pi:M\to \Rr$ is a submersion, one may consider   local coordinates $(t,q^i)$ on $M$ adapted to the submersion $\pi$   such that $\pi:M\to \Rr$ is the coordinate t. Denote by   $(t,p,q^i,p_i)$ (respect. $(t,q^i,p_i)$) the 
corresponding local coordinates on $T^*M$ (respect. on $V^*\pi$). With respect to them, we have that   
\[
   \pbr{t,q^i}_{V^*\pi}=\pbr{t,p_i}_{V^*\pi}=\pbr{q^i,q^j}_{V^*\pi}=\pbr{p_i,p_j}_{V^*\pi}=0, \quad
   \pbr{q^i,p_j}_{V^*\pi}=\delta^i_j,
\]
and 
\[
\psi_\pi(s, (t,p,q^i,p_i))=(t,s+p,q^i,p_i),\qquad \mu_\pi(t,p,q^i,p_i)=(t,q^i,p_i).
\]

If the local expression of a Hamiltonian section  $h:V^*\pi\to T^*M$ is given by 
\[
 h(t,q^i,p_i)=(t,-H(t,q,p),q^i,p_i),
\]
then, 
\[
 F_h(t,p,q^i,p_i)=p+H(t,q^i,p_i)
\]
and 
\[
 \omega_h = \d q^i\wedge\d p_i + \pd{H}{q^i}\d q^i\wedge\d t +\pd{H}{p_i}\d p_i\wedge\d t, \;\;\;\;\; \eta=\d t.
\]
Thus,  the Reeb vector field  $\Ree_h$ of the cosymplectic structure $(\omega_h,\eta)$ and the Hamiltonian vector field ${\mathcal H}_{F_h}$  have the following  local expressions 
\[
   \Ree_h  = \pd{}{t} + \pd{H}{p_i}\pd{}{q^i} - \pd{H}{q^i}\pd{}{p_i},\qquad 
   \Ha_{F_h} = \pd{}{t} - \pd{H}{t}\pd{}{p} + \pd{H}{p_i}\pd{}{q^i} - \pd{H}{q^i}\pd{}{p_i}.
\]
We remark that the integral curves of $\Ree_h$ are just \emph{the Hamilton equations} for $h$, 
\[
 \frac{\d q^i}{\d t} = \pd{H}{p_i}, \quad\quad \frac{\d p_i}{\d t} = - \pd{H}{q^i}, \qquad\mbox{for all }i.
\]

\section{The category of symplectic principal $\Rr$-bundles}

Motivated by the example of the above section, one may introduce the notion of a symplectic principal $\Rr$-bundle as follows. 

Let  $\mu:A\to V$  be a principal $\Rr$-bundle 
(a {\it AV-bundle} in the terminology  \cite{GrabGrabUrbAV}). We will denote by
\[
 \psi:\Rr \times A \to A,\qquad (s,a) \mapsto \psi_s(a),
\]
the corresponding principal action of the Lie group 
$(\Rr,+)$ on the manifold $A$.

In this case the vertical distribution of $\mu$ has 
dimension $1$ and it is generated by the infinitesimal 
generator $Z_\mu\in\X(A)$ whose flow is $\psi_s$.

\begin{defn}
We will say that $\mu:(A,\Omega)\to V$ is a 
\emph{symplectic principal $\Rr$-bundle}, if $\mu:A\to V$ is a
principal $\Rr$-bundle and $\Omega$ is a symplectic structure on $A$ such that 
the associated principal action 
$\psi:\Rr\times A\to A$ is symplectic.
\end{defn}

Note that the infinitesimal generator $Z_\mu$ 
of a symplectic principal $\Rr$-bundle $\mu:(A,\Omega)\to V$ is a locally  
Hamiltonian vector field.

\begin{ex}\label{StAVbundle}
 \emph{The standard symplectic principal $\Rr$-bundle 
associated with a fibration}. If $\pi:M\to\Rr$ is a surjective submersion then $T^*M$ is the total space of a symplectic principal $\Rr$-bundle 
over $V^*\pi$ (see Section \ref{section2}).
\hfill{$\diamond$}
\end{ex} 

\begin{ex}\label{StAVbundleB} \emph{The standard symplectic principal $\Rr$-bundle 
associated with a fibration and  a magnetic term.} Let $\pi:M\to \Rr$ be a surjective submersion with total space a manifold $M$ 
of dimension $n+1$ 
and $\beta$ a closed $2$-form on $M$. 
Consider the closed basic $2$-form (\emph{the magnetic term}) $B=\pi_M^*\beta$ on $T^*M$, where 
$\pi_M:T^*M\to M$ is the canonical projection. An easy computation shows 
that $B$ is invariant with respect to the $\Rr$-principal action 
of the standard symplectic principal $\Rr$-bundle $\mu_\pi:T^*M\to V^*\pi$. Thus, if 
$\Omega_M$ is the canonical symplectic form on $T^*M$, 
$\mu_\pi:(T^*M,\Omega_M-B)\to V^*\pi$ is a symplectic principal $\Rr$-bundle.

\hfill{$\diamond$}
\end{ex}
Now, we will prove a version of Darboux Theorem for a symplectic principal $\Rr$-bundle.
\begin{thm}\label{locExpr}
 Let $\mu:(A,\Omega)\to V$ be a symplectic principal $\Rr$-bundle with in\-fi\-ni\-te\-si\-mal 
generator $Z_\mu$. 
Suppose that $\dim A=2n+2$. Then, for any $a\in A$, there 
exist local coordinates $(t,p,q^i,p_i)$, ($i=1,\dots,n$) 
in a neighborhood $U$ of $a$ such that
\begin{itemize}
 \item[i)] the local expression of $\mu:A\to V$ is
           \begin{equation}\label{LocalExpressionMu}
             \mu(t,p,q^i,p_i)=(t,q^i,p_i),
           \end{equation}
 \item[ii)] $(t,p,q^i,p_i)$ are Darboux coordinates 
           for $\Omega$.
\end{itemize}
Moreover, the local expression of the infinitesimal 
generator is $Z_\mu=\pd{}{p}$.
\end{thm}
\begin{proof}
The proof is based on the well-known construction of the Darboux coordinates (see, for instance,  \cite{Ar,LM}). Fix $a\in A$. Since the vector field $Z_\mu$ is locally 
Hamiltonian, there exists a local function $t$ such that $Z_\mu=-\Ha_t$. Choose a function
$p$ (eventually defined on a smaller open neighborhood of $a$) 
such that $Z_\mu(p)=1$. Using the generalized Darboux Theorem on the closed $2$-form $\Omega'=\Omega-\d t\wedge\d p$ of rank $2n$, one may complete $t,p$ to a coordinate system $(t,p,q^i,p_i)$ such that
\[
 \Omega=\d t\wedge\d p + \sum_i\d q^i\wedge\d p_i.
\]
It follows that $Z_\mu=-\Ha_t=\pd{}{p}$. Since $\mu$ is  the 
projection of $A$ on $A/\langle Z_\mu\rangle$, the local expression of 
$\mu$ is as in \eqref{LocalExpressionMu}. 
\end{proof}

We will say that $(t,p,q^i,p_i)$ in the previous theorem
are \emph{canonical coordinates} for the symplectic 
principal $\Rr$-bundle $\mu$.

If $\mu:(A,\Omega)\to V$ is a symplectic principal $\Rr$-bundle, 
then, using a well-known result on Poisson reduction (see for instance \cite{OrRa}, Theorem 10.1.1) we have that  the base manifold $V$ may be canonically equipped with a Poisson structure as we show in the following result.
 
\begin{prop}\label{PoissonInducedByMu}
 Let $\mu:(A,\Omega)\to V$ be a symplectic principal $\Rr$-bundle. 
Then, there exists a unique Poisson structure $\pbrEm_V$ 
on $V$ such that $\mu$ is a Poisson map, i.e.
\begin{equation}\label{muPoissonMap}
   \pbr{ f\circ\mu , f'\circ\mu }_A = 
   \pbr{ f , f' }_V \circ\mu,
    \qquad\mbox{for any }f,f'\in C^\infty(V),
\end{equation}
where $\pbrEm_A$ is the Poisson bracket on $A$ induced by the symplectic form $\Omega$.
\end{prop}

Note that for canonical coordinates of $\mu$, if  $(t,q^i,p_i)$ are the induced coordinates on $V$,  the 
corresponding local expression of the 
Poisson bracket on $V$ with respect to these coordinates 
is the following one:
\[
   \pbr{t,q^i}_V=\pbr{t,p_i}_V=\pbr{q^i,q^j}_V=\pbr{p_i,p_j}_V=0, \qquad
   \pbr{q^i,p_j}_V=\delta^i_j.
\]


\begin{ex}
In the particular case when we have  an standard symplectic principal $\Rr$-bundle $\mu_\pi:(T^*M,\Omega_M)\to V^*\pi$ associated with a fibration $\pi:M\to \Rr,$ the Poisson structure is just the one described in (\ref{linearPoissonVpi}).

Additionaly, we suppose that we have a 
 closed $2$-form   $\beta$ on $M$. Denote by $B=\pi_M^*(\beta)\in \Omega^1(T^*M)$, where $\pi_M:T^*M\to M$ is the canonical projection,  and by 
\begin{itemize}
 \item $\Lambda_{T^*M}$ and $\Lambda_{T^*M}^B$ the Poisson 
       structures on $T^*M$ induced by the symplectic $2$-form $\Omega_M$ 
       and $\Omega_M-B$, respectively;
 \item $\Lambda_{V^*\pi}$ and $\Lambda_{V^*\pi}^B$ the Poisson 
       structures on $V^*\pi$ induced on the base space of the symplectic principal $\Rr$-bundles 
       $\mu_\pi:(T^*M,\Omega_M)\to V^*\pi$ and $\mu_\pi:(T^*M,\Omega_M-B)\to V^*\pi$, 
       respectively.
\end{itemize}

If $\pbrEm_{T^*M}$ and  $\pbrEm^B_{T^*M}$ are the Poisson brackets on $T^*M$
induced by the symplectic forms $\Omega_M$ and $\Omega_M-B$, respectively,
then, one may easily prove that 
\begin{equation}\label{DeformedPoissBrack1}
 \pbr{F, F'}^B_{T^*M} = \pbr{F, F'}_{T^*M} + B(\Ha_F, \Ha_{F'}),
\end{equation}
for any $F,F'\in C^\infty(T^*M)$, where $\Ha_F, \Ha_{F'}\in\X(T^*M)$ 
are the Hamiltonian vector fields of $F,F'$ with respect to the 
symplectic structure $\Omega_M$ (see \cite{MarsEtAl}).

Alternatively,  \eqref{DeformedPoissBrack1} may be written in terms of vertical lift of $\beta.$ We recall that, for a vector bundle $\tau:E\to Q$,  the vertical lift  $\gamma^v$ of a section $\gamma$ of $\wedge^pE\to Q$ 
 is a $p$-vector on $E$. In fact, if $(q^i)$ are local coordinates on an open subset $U$ of $Q$, $\{e_\alpha\}$ is a local basis of $\Gamma(E)$ and $\gamma=\gamma^{i_1\dots i_p}e_{i_1}\wedge \dots \wedge e_{i_p}\mbox{ on } U$
then
\[
 \gamma^v=\gamma^{i_1\dots i_p}\frac{\partial}{\partial y^{i_1}}\wedge \dots \wedge \frac{\partial}{\partial y^{i_p}}
\]
where $(q^i,y^\alpha)$ are the corresponding local coordinates on $E$. 
Then,  \eqref{DeformedPoissBrack1} may be rewritten as 
\begin{equation}\label{DeformedPoissBrack}
 \pbr{F, F'}^B_{T^*M} = \pbr{F, F'}_{T^*M} + \beta^v(\d F, \d F').
\end{equation} 
Indeed, it is sufficient to prove that if $F$ and $F'$ are linear or basic functions 
on $T^*M$ then $\beta^v(\d F, \d F')=B(\Ha_F, \Ha_{F'})$. 

Therefore, from \eqref{DeformedPoissBrack}, we deduce that the Poisson structures
$\Lambda_{T^*M}^B$ and $\Lambda_{T^*M}$ are related as follows
\begin{equation}\label{r1}
 \Lambda^B_{T^*M}=\Lambda_{T^*M}+\beta^v.
\end{equation}

On the other hand, we consider the section $\bar{\beta}$ of the vector bundle $\wedge^2V^*\pi\to M$  defined by  
\[
 \bar\beta(x)=\beta(x)_{|V_x\pi\times V_x\pi},
 \qquad\mbox{for any }x\in N.
\]
If $\pbrEm^B_{V^*\pi}$ and $\pbrEm_{V^*\pi}$ denote the 
Poisson brackets on $V^*\pi$ induced by $\Lambda^B_{V^*\pi}$ and 
$\Lambda_{V^*\pi}$, respectively, from \eqref{DeformedPoissBrack} and Proposition 
\ref{PoissonInducedByMu}, we have that
\begin{equation}\label{DeformedPoissBrackV}
 \pbr{f,f'}^B_{V^*\pi}\circ\mu_\pi 
  = \pbr{f,f'}_{V^*\pi}\circ\mu_\pi 
    +\beta^v(\mu_\pi^*(\d f),\mu_\pi^*(\d f')), 
\end{equation}
for $f,f'\in C^\infty(V^*\pi)$.

Moreover, one may easily prove that
\[
 \beta^v(\mu_\pi^*(\d f),\mu_\pi^*(\d f')) = \bar\beta^v(\d f,\d f') \circ\mu_\pi.
\]
Thus, 
\begin{equation}\label{r2}
 \Lambda^B_{V^*\pi}=\Lambda_{V^*\pi}+\bar\beta^v.
\end{equation}
\end{ex}


In the last part of this section we will study morphisms 
in the category of symplectic principal $\Rr$-bundles.

 Let $\mu:A\to V$ and $\mu':A'\to V'$ be two principal $\Rr$-bundles with 
principal actions $\psi$ and $\psi'$, respectively. Suppose that the  
function $\varphi:A\to A'$ is  a principal $\Rr$-bundle 
morphism, that is, $\varphi$ is equivariant with respect to the 
principal actions, i.e.
\begin{equation}\label{varphiEquivariant}
 \varphi\circ\psi_s=\psi'_s\circ\varphi,
  \qquad\mbox{for any }s\in\Rr.
\end{equation}

From \eqref{varphiEquivariant}, one deduces that the infinitesimal generators $Z_\mu$ and $Z_{\mu'}$ of $\mu:A\to V$ and 
$\mu':A'\to V'$ respectively, are $\varphi$-related. Moreover, by passing 
to the quotient and using \eqref{varphiEquivariant}, one may define a 
map $\varphi^V:V\to V'$ characterized by the following relation
\begin{equation}\label{InducedVarphiV}
 \mu'\circ\varphi=\varphi^V\circ\mu.
\end{equation}
Note that, since $\mu$ is a submersion, then $\varphi^V$ is smooth. The following diagram 
illustrates the situation
\[
 \xymatrix{
   A \ar@{->}[r]^{\varphi} \ar@{->}[d]_{\mu} & A' \ar@{->}[d]^{\mu'} \\
   V \ar@{->}[r]^{\varphi^V} & V' 
          }
\]
Now, suppose that $\varphi$ is a  principal $\Rr$-bundle embedding. Then, $\varphi^V$ is 
also an embedding. In fact, using \eqref{varphiEquivariant}, 
\eqref{InducedVarphiV} and the fact that $\mu\circ\psi_s=\mu$, for all $s$, 
we deduce that $\varphi^V$ is an injective immersion. Moreover, standard topological 
arguments show that the map $\varphi^V:V\to\varphi^V(V)$ is an homeomorphism. 

On the other hand, if $\varphi:A\to A'$ is a diffeomorphism, then $\varphi^V$ also is a diffeomorphism. 
Indeed, 
\[
 (\varphi^V)^{-1}=(\varphi^{-1})^{V'}.
\]

\begin{defn}
 Let $\mu:(A,\Omega)\to V$ and $\mu':(A',\Omega')\to V'$ be symplectic 
 principal $\Rr$-bundles. A smooth function $\varphi:A\to A'$ is said to be a 
\emph{symplectic  principal $\Rr$-bundle morphism} if $\varphi$ is a  principal $\Rr$-bundle 
morphism such that $\varphi^*\Omega'=\Omega$.
\end{defn}
If we change the word ``morphism'' by ``embedding'' (resp., ``isomorphism'') 
in the previous definition, we obtain the notion of a \emph{symplectic  principal $\Rr$-bundle 
embedding} (resp., a \emph{symplectic  principal $\Rr$-bundle isomorphism}).
In what follows, we related the Poisson structures induced on a symplectic principal $\Rr$-bundle embedding 
 
Firstly,  we remark that in general, if $\varphi:(A,\Omega)\to (A',\Omega)$ is a symplectic morphism, $\varphi$ is not a Poisson morphism with respect to the corresponding Poisson structures $\Lambda_{A}$ and $\Lambda_{A'}$ (see, for instance, \cite{OrRa}). In fact, if $\varphi:(A,\Omega)\to (A',\Omega)$ is a symplectic embedding, then one may give a relation between the corresponding Poisson structures $\Lambda_{A}$ and $\Lambda_{A'}$  on $A$ and $A'$, respectively. Under the identification of  the tangent space $T_aA$ at  $a\in A$ with $T_a\varphi(T_aA),$ we have 
\begin{equation}\label{DirectSumDual}
 T^*_{\varphi(a)}A'=(T_aA)^o\oplus((T_aA)^{\Omega'})^o=(T_aA)^o\oplus(\#_{\Lambda_{A'}}(T_aA)^o)^o
\end{equation}
where $(T_aA)^{\Omega'}$ denotes the symplectic orthogonal space of 
$T_a\varphi(T_aA)\cong T_aA$ with respect to the symplectic form $\Omega'_{\varphi(a)}$ on $T_{\varphi(a)}A'$ and $W^o$ denotes the 
annihilator  of a subspace $W\subset T_{\varphi(a)}A'$ in $T^*_{\varphi(a)}A'$. 

Denote by  $\widetilde{P}_a:T^*_{\varphi(a)}A'\to ((T_aA)^{\Omega'})^o$ and $\widetilde{Q}_a:T^*_{\varphi(a)}A'\to(T_aA)^o$ the corresponding  projectors  associated with the splitting 
\eqref{DirectSumDual}. Note that if $\alpha\in(T_aA)^o$, then  
 $(T^*_a\varphi)(\alpha)=0$ and $\widetilde{P}_a(\alpha)=0.$
Using these facts, we deduce that  the Poisson structures $\Lambda_A$ and $\Lambda_{A'}$  are related as follows
\begin{equation}\label{LambdaRelation}
  \Lambda_A(a) ( \varphi^*(\alpha'_1), \varphi^*(\alpha'_2)) 
       = \Lambda_{A'}(\varphi(a)) ( \alpha'_1 , \alpha'_2 )
        - \Lambda_{A'}(\varphi(a)) ( \widetilde{Q}_a(\alpha'_1) , \widetilde{Q}_a(\alpha'_2) )
\end{equation} for all $\alpha'_1, \alpha'_2\in T^*_{\varphi(a)}A'$.

Now, let $\varphi:A\to A'$ be an embedding of the  principal $\Rr$-bundles 
$\mu:(A,\Omega)\to V$ and $\mu':(A',\Omega')\to V'$ and let 
$\varphi^V:V\to V'$ be the corresponding embedding between the base spaces $V$ and $V'$. Using \eqref{InducedVarphiV}, \eqref{DirectSumDual} 
and the fact that the 
infinitesimal generators $Z_\mu$ and $Z_{\mu'}$ are 
$\varphi$-related, one may obtain that 
\[
 T_{\varphi^V(v)}V'=T_vV\oplus T_{\varphi(a)}\mu'((T_aA)^{\Omega'}).
\]
Moreover, from \eqref{InducedVarphiV} and since $\dim(T_vV)^o=\dim(T_aA)^o$, it 
follows that 
\[
 (T_aA)^o=T^*_{\varphi(a)}\mu((T_vV)^o),
\]
$(T_vV)^o$ being the annihilator of $T_vV$ in $T_{\varphi^V(v)}V'$. Therefore, 
from 
the fact that $\mu'$ is a Poisson map, we deduce that 
\[
 T_{\varphi(a)}\mu'((T_aA)^{\Omega'})=T_{\varphi(a)}\mu'(\sharp_{\Lambda_{A'}}((T_aA)^o))=\sharp_{\Lambda_{V'}}(T_vV)^o
\]
where $\Lambda_{V'}$ is the Poisson structure on $V'$ induced by the symplectic 
 principal $\Rr$-bundle $\mu':(A',\Omega')\to V'$. So, we may again consider the splittings 
\begin{gather*}
 T^*_{\varphi^V(v)}V'=(T_vV)^o\oplus \big(\sharp_{\Lambda_{V'}}(T_vV)^o\big)^o
\end{gather*}
and the first projector $\widetilde{q}_v:T^*_{\varphi^V(v)}V'\to(T_vV)^o.$

Now, from \eqref{muPoissonMap},  \eqref{InducedVarphiV}
and \eqref{LambdaRelation} and the relation $\widetilde{Q}_a\circ T^*_{\varphi(a)}\mu'=T^*_{\varphi(a)}\mu'\circ\widetilde{q}_v$, we obtain
\begin{prop}\label{PoissonEmbedding}
 Let $\varphi:A\to A'$ be an embedding of the symplectic principal $\Rr$-bundles 
$\mu:(A,\Omega)\to V$ and $\mu':(A',\Omega')\to V'$ and let 
$\varphi^V:V\to V'$ be the corresponding embedding between the base spaces $V$ and $V'$. 
Then, the Poisson structures $\Lambda_V$ and $\Lambda_{V'}$  induced on $V$ and $V'$ 
           respectively, by $\mu$ and $\mu'$,  are related by
\[
 \begin{split}
  \Lambda_V(v) ( (\varphi^V)^*(\sigma'_1), (\varphi^V)^*(\sigma'_2 )) 
   &= \Lambda_{V'}(\varphi^V(v)) (\sigma'_1 , \sigma'_2 ) 
         - \Lambda_{V'}(\varphi^V(v)) ( \widetilde{q}_v(\sigma'_1) , \widetilde{q}_v(\sigma'_2) )
 \end{split}
\]
            with $v\in V$ and $\sigma'_1, \sigma'_2\in T^*_{\varphi^V(v)}V'$. If $\varphi:A\to A'$ is an isomorphism of  principal $\Rr$-bundles, then $\varphi^V:V\to V'$ is a Poisson isomorphism.  
\end{prop}

\section{Reduction of symplectic principal $\Rr$-bundles}
In this section we describe the reduction process of a symplectic principal $\Rr$-bundle in the presence of a  momentum map. 

\subsection{Canonical actions and momentum maps}

In this subsection, we consider a special type of actions which are compatible 
with the symplectic principal $\Rr$-bundle in a certain sense.
\begin{defn} 
We say that an action $\phi:G\times A\to A$ is a 
\emph{canonical action} on the symplectic  principal $\Rr$-bundle 
$\mu:(A,\Omega)\to V$ with principal action $\psi:\Rr\times A\to A$, if the following conditions hold:
\begin{itemize}
 \item[\emph{i)}]   $\phi$ is a symplectic action,
 \item[\emph{ii)}]  the actions $\psi$ and $\phi$ commute, that is 
                    \begin{equation}\label{PhiPsiCompatible}
                      \phi_g\circ\psi_s=\psi_s\circ\phi_g,
                      \qquad\mbox{for any }g\in G,\,s\in\Rr, 
                    \end{equation}
 \item[\emph{iii)}] the $1$-form $\zeta_\mu=i_{Z_\mu}\Omega$
                    is basic with respect to $\phi$, 
                    i.e.~$\zeta_\mu(\xi_A)=0$ for 
                    any $\xi\in\g$, where $\xi_A$ is the 
                    infinitesimal generator of $\phi$ defined 
                    by $\xi$.
\end{itemize}
\end{defn}
We will see that,  for a canonical 
action on $\mu$ with momentum map,  one 
may induce canonically a Poisson action  with a momentum map on $V$ (for the definition of a momentum map associated with the Poisson action, see Appendix \ref{MWRed}). In fact, let $\xi$ be an element of the Lie algebra $\g$.
Using  that $\zeta_\mu$ is basic, it follows that 
\begin{equation}\label{JxiBasic}
 Z_\mu(J_\xi)=i_{\xi_A}\Omega(Z_\mu)=-\zeta_\mu(\xi_A)=0.
\end{equation}
Therefore, the function 
$J_\xi:A\to\Rr$ is constant on the fibers of $\mu$ and thus, 
\begin{equation}\label{JpsiInvariant}
 J\circ\psi_s=J, \qquad\mbox{for any $s$.}
\end{equation}

Using this fact and \eqref{PhiPsiCompatible}, we may 
define the action $\phi^V:G\times V\to V$ of $G$ on $V$ and 
the map $J^V:V\to\g^*$ characterized by
\begin{align}
  \phi^V_g\circ\mu &= \mu\circ\phi_g,\qquad\mbox{for any }g\in G, \label{muEquivariant}\\
  J^V\circ\mu &= J \label{JVandJ}.
\end{align}
Note that, by construction, $\mu$ is equivariant with respect to the actions $\phi$ and $\phi^V$. 
So, $\mu$ transforms the infinitesimal generator 
$\xi_A\in\X(A)$ of $\xi\in\g$ with respect to the action 
$\phi$ into the infinitesimal generator $\xi_V\in\X(V)$ 
of $\xi$ with respect to the action $\phi^V$, that is,
\begin{equation}\label{TmuInfGenerator}
 T_a\mu(\xi_A(a))=\xi_V(\mu(a)),\qquad\mbox{for any }a\in A.
\end{equation}
Moreover, we have
\begin{prop} If $\phi:G\times A\to A$ is a canonical action equipped 
with a momentum map $J:A\to\g^*$, then: 
\begin{itemize}
 \item[\emph{i)}]  $\phi^V:G\times V\to V$ is a Poisson action;
 \item[\emph{ii)}] $J^V:V\to\g^*$ is a momentum map associated with
            $\phi^V$ and, if $J$ is $\Ad^*$-equivariant, then 
            so is $J^V$.
\end{itemize}
\end{prop}

\begin{proof}
For any $g\in G$, $\phi_g^V:V\to V$ is just the map induced by the symplectic principal $\Rr$-bundle isomorphism $\phi_g:A\to A$. Thus, using  Proposition \ref{PoissonEmbedding}, it follows that $\phi^V_g$ is a Poisson map.

Now, we prove that $J^V$ is a momentum map, that is, 
$\xi_V=\Ha_{J^V_\xi}$, for any $\xi\in\g$. In fact, for any 
$f\in C^\infty(V)$ and $a\in A$, one has, from 
\eqref{muPoissonMap}, \eqref{JVandJ} and \eqref{TmuInfGenerator}, that
\begin{equation*}
 \begin{split}
  \left(\xi_V(f)\right)(\mu(a)) &=
  \xi_A(a)(f\circ\mu) =\Ha_{J_\xi}(a)(f\circ\mu)=
 \pbr{f\circ\mu,J^V_\xi\circ\mu}_A(a) \\&= \pbr{f,J^V_\xi}_V(\mu(a)) = \Ha_{J^V_\xi}(f))(\mu(a)). 
 \end{split}
\end{equation*}
Since $a$ is an arbitrary element of $A$ and $\mu$ is surjective, we obtain that 
$\xi_V=\Ha_{J^V_\xi}$.

If $J$ is $\Ad^*$-equivariant (see Appendix \ref{MWRed}), then for any $v=\mu(a)\in V$ and for 
any $g\in G$
\[
 \Ad^*_{g^{-1}}(J^V(v))=\Ad^*_{g^{-1}}(J(a))=J(\phi_g(a))=J^V(\phi^V_g(v)).
\]
Thus, $J^V$ is $\Ad^*$-equivariant.
\end{proof}

\subsection{The reduction process of  symplectic principal $\Rr$-bundles}\label{AVred}

In this subsection, we will use the results of Appendix 
\ref{MWRed} to reduce a symplectic  principal $\Rr$-bundle 
equipped with a canonical action and an $\Ad^*$-equivariant 
momentum map.

Suppose that $\mu:(A,\Omega)\to V$ is a symplectic  principal $\Rr$-bundle 
equipped with a canonical action 
$\phi:G\times A\to A$ of a Lie group $G$ with an $\Ad^*$-equivariant 
momentum map $J:A\to\g^*$. One may induce a Poisson action 
$\phi^V:G\times V\to V$ on $V$ with an $\Ad^*$-equivariant 
momentum map $J^V:V\to\g^*$. Assume that $\phi^V$ is free 
and proper. Then, so is $\phi$.

If $\nu\in\g^*$, from Marsden-Weinstein reduction Theorem 
(resp., Poisson reduction Theorem), we may induce a reduced symplectic 
structure $\Omega_\nu$ (resp., a reduced Poisson bracket $\pbrEm_\nu$) on the 
quotient space $A_\nu=J^{-1}(\nu)/G_\nu$ 
(resp., $V_\nu=(J^V)^{-1}(\nu)/G_\nu$). Let's prove that $A_\nu$ 
and $V_\nu$ are the total space and the base manifold,
respectively, of a reduced  principal $\Rr$-bundle $\mu_\nu:A_\nu\to V_\nu$.

The map $\mu_\nu:A_\nu\to V_\nu$ is defined as follows. Using \eqref{JVandJ}, it follows that 
the restriction $\mu:J^{-1}(\nu)\to(J^V)^{-1}(\nu)$
of $\mu$ to the closed submanifold $J^{-1}(\nu)$ is a surjective
submersion. Moreover, we have that the actions 
$\phi:G_\nu\times J^{-1}(\nu)\to J^{-1}(\nu)$ and 
$\phi^V:G_\nu\times (J^V)^{-1}(\nu)\to (J^V)^{-1}(\nu)$ of the 
isotropy group $G_\nu$ on $J^{-1}(\nu)$ and $(J^V)^{-1}(\nu)$ 
respectively, are free and proper and $\mu$ is equivariant with respect 
to them. Denote by 
\[
 \mu_\nu:A_\nu=J^{-1}(\nu)/G_\nu\to V_\nu=(J^V)^{-1}(\nu)/G_\nu
\]
the induced map on the quotient spaces which is characterized by
\begin{equation}\label{muAndPi}
 \mu_\nu\circ\pi_\nu=\pi^V_\nu\circ\mu,
\end{equation}
where $\pi_\nu:J^{-1}(\nu)\to A_\nu$ and 
$\pi^V_\nu:(J^V)^{-1}(\nu)\to V_\nu$ are the corresponding 
canonical projections. Note that $\mu_\nu$ is a surjective 
submersion.
 
Moreover, using \eqref{PhiPsiCompatible} and \eqref{JpsiInvariant}, we have that the principal action $\psi:\Rr\times A\to A$ 
restricts to an action of $\Rr$ on $J^{-1}(\nu)$. So, it defines an action of $\Rr$ on $A_\nu$, $\psi_\nu:\Rr\times A_\nu\to A_\nu$ characterized by
\begin{equation}\label{ReducedAction}
 (\psi_\nu)_s\circ\pi_\nu = \pi_\nu\circ\psi_s.
\end{equation}

In addition, we may prove the following result. 
\begin{thm}\label{AVreductionTheorem}
Let $\mu:(A,\Omega)\to V$ be a symplectic principal $\Rr$-bundle 
equipped with a canonical action $\phi:G\times A\to A$ 
and an $\Ad^*$-equivariant momentum map $J:A\to\g^*$. 
Suppose that the induced action $\phi^V:G\times V\to V$ 
is free and proper. Then, for any $\nu\in\g^*$, 
$\mu_\nu: (A_\nu,\Omega_\nu)\to V_\nu$ is a symplectic 
 principal $\Rr$-bundle with principal action defined by 
\eqref{ReducedAction}, where $\Omega_\nu$ is the reduced symplectic structure 
on the reduced space $A_\nu=J^{-1}(\nu)/G_\nu$. 
Moreover, the restriction of the infinitesimal generator 
$Z_\mu$ of $\mu$ to $J^{-1}(\nu)$ is tangent to 
$J^{-1}(\nu)$ and $\pi_\nu$-projectable. Its 
$\pi_\nu$-projection is the infinitesimal generator 
$Z_{\mu_\nu}$ of $\mu_\nu$.
\end{thm}

\begin{proof}
First of all, we will see that $\psi_\nu$ is a free action. Indeed, suppose that 
$(\psi_\nu)_s(\pi_\nu(a))=\pi_\nu(a)$, for $a\in J^{-1}(\nu)$. Then, from \eqref{PhiPsiCompatible} 
and \eqref{ReducedAction}, we deduce that there exists $g\in G_\nu$ such that
\begin{equation}\label{FormAVRedTheorem}
 a=\psi_s(\phi_g(a)).
\end{equation}
This implies that 
\[
 \mu(a)=\mu(\psi_s(\phi_g(a)))=\mu(\phi_s(a))
\]
and, using  
\eqref{muEquivariant}, it follows that $\phi^V_g(\mu(a))=\mu(a)$. Thus, since $\phi^V$ is a free action, we obtain that $g=e$. Therefore, from \eqref{FormAVRedTheorem}, we conclude that $s=0$.

Next, we will prove that the fibers of $\mu_\nu$ are just the orbits of the action of $\Rr$ on $A_\nu=J^{-1}(\nu)/G_\nu$. In other words, we will see that 
\[
 (\psi_\nu)_{\pi_\nu(a)}(\Rr)=(\mu_\nu)^{-1}(\mu_\nu(\pi_\nu(a))), \qquad\mbox{for }a\in J^{-1}(\nu).
\]
In fact, a straightforward computation, using \eqref{muEquivariant}, \eqref{muAndPi} and \eqref{ReducedAction}, proves the result. Consequently, $\mu_\nu:A_\nu\to V_\nu$ is a  principal $\Rr$-bundle. 

Now, as we know, the action $\psi:\Rr\times A\to A$ 
restricts to an action of $\Rr$ on $J^{-1}(\nu)$. This implies that the restriction to $J^{-1}(\nu)$ of the infinitesimal generator 
$Z_\mu$ of $\mu$ is tangent to $J^{-1}(\nu)$ and ${Z_\mu}_{|J^{-1}(\nu)}$ is just the infinitesimal generator of the action of $\Rr$ on $J^{-1}(\nu)$. 

In addition, since the projection $\pi_\nu$ is equivariant, we obtain that ${Z_\mu}_{|J^{-1}(\nu)}$ is 
$\pi_\nu$-projectable and its projection is the infinitesimal generator $Z_{\mu_\nu}$ of $\mu_\nu$.

Finally, we prove that $Z_{\mu_\nu}$ is a locally Hamiltonian vector field. We will show that the flow 
$(\psi_\nu)_s:A_\nu\to A_\nu$ of $Z_{\mu_\nu}$ preserves the symplectic form $\Omega_\nu$. In fact, using \eqref{ReducedAction}, \eqref{ReducedSymplecticForm} (see Appendix \ref{MWRed}) and the invariance of $\Omega$ under the action of $\psi_s$, we get
\[
 \pi^*_\nu\left((\psi_\nu)_s^*\Omega_\nu\right) = 
 \psi_s^*(\pi^*_\nu\Omega_\nu) = \psi_s^*(i^*_\nu\Omega) = i^*_\nu\Omega = \pi^*_\nu\Omega_\nu,
\]
As a consequence, we have that $(\psi_\nu)_s^*\Omega_\nu=\Omega_\nu$.
\end{proof}

From Proposition \ref{PoissonInducedByMu}, the symplectic 
$2$-form $\Omega_\nu$ on $A_\nu$ induces a Poisson structure 
$\pbrEm_{V_\nu}$ on $V_\nu$. On the other hand, using Theorem \ref{PoissonReduction}, we have that $V_\nu$ 
is equipped with a reduced Poisson structure. The following result proves that these structures are the 
same one.
\begin{prop}\label{twoPoissStructEqual}
Under the same hypotheses as in Theorem \ref{AVreductionTheorem}, 
the reduced Poisson bracket $\pbrEm_\nu$ on $V_\nu$ 
is just the one induced by the symplectic principal $\Rr$-bundle 
$\mu_\nu:A_\nu\to V_\nu$.
\end{prop}
\begin{proof}
Let $f_\nu,f'_\nu$ be functions on $V_\nu$ and 
$\pi^V_\nu(v)\in V_\nu$, with $v\in(J^V)^{-1}(\nu)$. 
Choose $a\in J^{-1}(\nu)$ such that $\mu(a)=v$.
The bracket $\pbrEm_\nu$ is characterized by
\[
 \pbr{ f_\nu , f'_\nu }_\nu(\pi^V_\nu(v)) = 
 \pbr{ f , f' }_V(v)
\]
where $f,f'\in C^\infty(V)$ are arbitrary 
$G$-invariant extensions of $f_\nu\circ\pi^V_\nu$ 
and $f'_\nu\circ\pi^V_\nu$, respectively.

Note that $f\circ\mu,f'\circ\mu\in C^\infty(A)$ are 
$G$-invariant extensions of $f_\nu\circ\pi^V_\nu\circ\mu_{|J^{-1}(\nu)}$ 
and $f'_\nu\circ\pi^V_\nu\circ\mu_{|J^{-1}(\nu)}$, respectively. 
Applying Theorem \ref{SymplecticReduction} (see Appendix \ref{MWRed}), we 
obtain that the Poisson bracket $\pbrEm_{A_\nu}$ on $A_\nu$ 
induced by $\Omega_\nu$ may be expressed as follows
\[
 \pbr{ f_\nu\circ\mu_\nu , f'_\nu\circ\mu_\nu}_{A_\nu}(\pi_\nu(a)) = 
 \pbr{ f\circ\mu , f'\circ\mu }_A(a).
\]
Therefore, using \eqref{muPoissonMap} refered to $\mu$ and $\mu_\nu$, we have
\[
 \begin{split}
  \pbr{f_\nu,f'_\nu}_{V_\nu}(\pi^V_\nu(v)) 
  &= \pbr{ f_\nu\circ\mu_\nu , f'_\nu\circ\mu_\nu }_{A_\nu}(\pi_\nu(a)) \\
  &= \pbr{ f\circ\mu , f'\circ\mu}_A(a) = \pbr{ f_\nu , f'_\nu }_\nu(\pi^V_\nu(v)).
 \end{split}
\]
This proves that $\pbr{f_\nu,f'_\nu}_{V_\nu}=\pbr{f_\nu,f'_\nu}_\nu$.
\end{proof}

\subsection{The standard case}
 In this subsection, we want to apply the reduction procedure to the standard symplectic principal $\Rr$-bundle $\mu_\pi:(T^*M,\Omega_M)\to V^*\pi$ associated with a surjective submersion $\pi:M\to\Rr$ (see Section \ref{section2} and Example \ref{StAVbundle}), where $\Omega_M$ is the canonical symplectic structure 
on $T^*M$. 
 
Suppose that $\phi:G\times M\to M$ is an action of a connected Lie group $G$ on the manifold $M$. 
The lifted action $T^*\phi:G\times T^*M\to T^*M$ is symplectic with respect to the 
standard symplectic structure $\Omega_M$ on $T^*M$ and it admits an 
$\Ad^*$-equivariant momentum map $J:T^*M\to\g^*$ 
given by 
\begin{equation}\label{Jcotangent}
 J(\alpha_x)(\xi)=J_\xi(\alpha_x)=\alpha_x(\xi_M(x)),
  \qquad\mbox{for any }\xi\in\g
\end{equation}
where $\xi_M\in\X(M)$ is the infinitesimal generator of $\phi$ 
associated with $\xi$.

The following result gives a sufficient condition for $T^*\phi$ 
to be a canonical action on the standard symplectic principal $\Rr$-bundle $\mu_\pi$.
\begin{prop}
Let $\pi:M\to\Rr$ be a surjective submersion. Denote by $\mu_\pi:(T^*M,\Omega_M)\to V^*\pi$ the corresponding standard symplectic
 principal $\Rr$-bundle and by $T^*\phi:G\times T^*M\to T^*M$ the cotangent 
lift of an action $\phi:G\times M\to M$ of a connected Lie group $G$ on $M$. If $\pi$ is $G$-invariant, i.e.~$\pi\circ\phi_g=\pi$ for any $g\in G$, 
then $T^*\phi$ is a canonical action on $\mu_\pi$.
\end{prop}
\begin{proof}
Recall that the infinitesimal generator $\xi_{T^*M}$ of the 
action $T^*\phi$ associated to an element $\xi$ of $\g$ is just the 
Hamiltonian vector field of the linear function 
$\widehat{\xi}_M\in C^\infty(T^*M)$ associated with $\xi_M\in\X(M)$.

Moreover, since $\pi_M:T^*M\to M$ is equivariant with respect to 
the actions $T^*\phi$ and $\phi$, the vector fields 
$\xi_{T^*M}$ and $\xi_M$ are $\pi_M$-related. Now, using  that 
$Z_{\mu_\pi}$ is the Hamiltonian vector field 
of the function $-\pi\circ\pi_M$ and that $\pi$ is $G$-invariant, we get
\[
 \begin{split}
  \Omega_M(\xi_{T^*M},Z_{\mu_\pi})
    = \Ha_{\widehat\xi_M}(\pi\circ\pi_M)
   = \xi_{T^*M}(\pi\circ\pi_M) 
     = \xi_M(\pi)\circ\pi_M = 0.
 \end{split}
\]
Thus,  $\zeta_{\mu_\pi}=i_{Z_{\mu_\pi}}\Omega_M$ is basic. It follows also that
\[
 \begin{split}
  [Z_{\mu_\pi},\xi_{T^*M}] = - [\Ha_{\pi\circ\pi_M},\Ha_{\widehat\xi_M}] 
    = \Ha_{\pbr{\pi\circ\pi_M,\widehat\xi_M}_{T^*M}} = 0
 \end{split}
\]
for all $\xi\in\g$. Since $G$ is connected, the actions $\psi$ and $T^*\phi$ commute.
\end{proof}

Moreover, we note that if $\phi$ is free and proper, so is $(T^*\phi)^{V^*\pi}$.

The rest of the subsection is devoted to
give sufficient conditions for the reduced symplectic principal $\Rr$-bundle obtained from a standard 
principal $\Rr$-bundle to be again standard.  
We will use some well known results of the cotangent bundle reduction theory 
(see, for instance, \cite{MarsEtAl}). 

Suppose that a connected Lie group 
$G$ acts freely and properly on a manifold $M$. 

We assume that we have a $G$-invariant surjective submersion $\pi:M\to \Rr$.
Using Theorem \ref{AVreductionTheorem}, we obtain a reduced symplectic principal $\Rr$-bundle
\[
 (\mu_\pi)_\nu:((T^*M)_\nu,(\Omega_M)_\nu)\to (V^*\pi)_\nu,
\]
where $(T^*M)_\nu=J^{-1}(\nu)/G_\nu$ and $(V^*\pi)_\nu=(J^{V^*\pi})^{-1}(\nu)/G_\nu.$  

On the other hand, since $\pi$ is $G$-invariant, there exists a unique 
surjective submersion $\pi^*_{1,\nu}:M/G_\nu\to \Rr$ such that
\begin{equation}\label{defTildePi}
 \pi^*_{1,\nu}\circ\pi_{M,G_\nu}=\pi,
\end{equation}
where we have denoted by $\pi_{M,G_\nu}:M\to M/G_\nu$ the $\nu$-shape space.
Thus, we may consider the corresponding standard symplectic principal $\Rr$-bundle 
\[
 \mu_{\pi^*_{1,\nu}}:(T^*(M/G_\nu),\Omega_{M/G_\nu})\to V^*\pi^*_{1,\nu},
\]
$\Omega_{M/G_\nu}$ being the canonical symplectic $2$-form 
on the cotangent bundle $T^*(M/G_\nu)$.

We will prove that, under a suitable hypothesis, the reduced symplectic 
principal $\Rr$-bundle $(\mu_\pi)_\nu$ may be embedded into the standard symplectic principal $\Rr$-bundle $\mu_{\pi^*_{1,\nu}}$, 
where the total space $T^*(M/G_\nu)$ will be equipped with the canonical 
symplectic form $\Omega_{M/G_\nu}$ deformed by a \emph{magnetic term}.

The magnetic term is defined as follows. 
Consider the action $\phi_\nu:G_\nu\times M\to M$ deduced from 
$\phi:G\times M\to M$.
Its cotangent lift $T^*\phi_\nu:G_\nu\times T^*M\to T^*M$ has an 
$\Ad^*$-equivariant momentum map $J_\nu:T^*M\to\g_\nu^*$ obtained by 
restricting $J$, that is, for $\alpha_x\in T^*_xM$,
\begin{equation}\label{Jstep3}
 J_\nu(\alpha_x)=J(\alpha_x)_{|\g_\nu}.
\end{equation}
Let $\nu'=\nu_{|\g_\nu}\in\g_\nu^*$ be the restriction of 
$\nu$ to $\g_\nu$ . Since the actions are free and proper, 
$\nu$ and $\nu'$ are regular values for $J$ and $J_\nu$, respectively. 
Note that the inclusion of submanifolds 
$ \bar{\iota}:J^{-1}(\nu)\hookrightarrow J_\nu^{-1}(\nu') $ is a $G_\nu$-invariant embedding. 

We will use the following assumption
\begin{itemize}
 \item [\emph{(MT)}] There exists a $G_\nu$-invariant $1$-form 
$\lambda_\nu$ on $M$ with values in $J_\nu^{-1}(\nu')$.
\end{itemize}
In fact, if $A:TM\to \g$ is the connection $1$-form associated with a 
principal connection on the principal $G$-bundle $\pi_{M,G}:M\to M/G$ then $\lambda_\nu=\nu\circ A$ defines a $1$-form on $M$ which satisfies the condition 
\emph{(MT)}  (for more details, see \cite{MarsEtAl}).

Now, using that the $2$-form $\d\lambda_\nu$ is basic with respect to the projection $\pi_{M/G_\nu}$, we deduce that there exists a unique closed $2$-form
$\beta_{\lambda_\nu}$ on $M/G_\nu$ such that
\begin{equation*}\label{dConnexionForm}
 \pi_{M,G_\nu}^*\beta_{\lambda_\nu}=\d\lambda_\nu.
\end{equation*}

If we define the $2$-form $B_{\lambda_\nu}$ on $T^*(M/G_\nu)$ as \[
 B_{\lambda_\nu}=\pi_{M/G_\nu}^*\beta_{\lambda_\nu},
\]
where $\pi_{M/G_\nu}:T^*(M/G_\nu)\to M/G_\nu$
is the cotangent bundle projection, we may consider the corresponding standard symplectic principal $\Rr$-bundle 
\[
\mu_{\pi^*_{1,\nu}}:(T^*(M/G_\nu),\Omega_{M/G_\nu}-B_{\lambda_\nu})\to V^*\pi^*_{1,\nu}
\] 
with magnetic term $B_{\lambda_\nu}$ (see Example \ref{StAVbundleB}). 
The form $B_{\lambda_\nu}$ is usually called \emph{the magnetic term associated with $\lambda_\nu$.}

The main theorem of this section is the following one:
\begin{thm}\label{StandardReductionThm}
 Let $\phi:G\times M\to M$ be a free and proper action of a 
connected Lie group $G$ on the manifold $M$ and $\pi:M\to \Rr$ 
a $G$-invariant surjective submersion. 
Let $\nu\in\g^*$ and $\pi^*_{1,\nu}:M/G_\nu\to\Rr$ 
the surjective submersion obtained from $\pi$ by passing 
to the quotient. Choose a $G_\nu$-invariant $1$-form 
$\lambda_\nu\in\Omega^1(M)$ with values in $J_\nu^{-1}(\nu')$. Then there is 
a symplectic principal $\Rr$-bundle embedding
\[
 \varphi_{\lambda_\nu}:(T^*M)_\nu\to T^*(M/G_\nu)
\]
between the reduced symplectic principal $\Rr$-bundle 
$(\mu_\pi)_\nu:((T^*M)_\nu,(\Omega_M)_\nu)\to (V^*\pi)_\nu$ and
the standard symplectic principal $\Rr$-bundle 
$\mu_{\pi^*_{1,\nu}}:(T^*(M/G_\nu),\Omega_{M/G_\nu}-B_{\lambda_\nu})\to V^*\pi^*_{1,\nu}$, 
with the symplectic structure modified by $B_{\lambda_\nu}$ $\in\Omega^2(T^*(M/G_\nu))$, the magnetic term 
 associated with $\lambda_\nu$.

Moreover, $\varphi_{\lambda_\nu}$ is a symplectic principal $\Rr$-bundle isomorphism 
if and only if $\g=\g_\nu$ (in particular, if $\nu=0$ or 
$G=G_\nu$), where $\g_\nu$ is the Lie algebra of $G_\nu$. 
\end{thm}

\begin{proof}
 Using the cotangent bundle reduction theory (see \cite{AbrMarsd} for more details), 
we have that there is a symplectic embedding 
\[
 \varphi_{\lambda_\nu}:(T^*M)_\nu\to T^*(M/G_\nu)
\]
which is an isomorphism if and only if $\g=\g_\nu$.

Now, we will prove that $\varphi_{\lambda_\nu}$ is a principal $\Rr$-bundle 
morphism between $(\mu_\pi)_\nu$ and $\mu_{\pi^*_{1,\nu}}$.

Firstly, we suppose that $G=G_\nu$. In such a case, $\varphi_{\lambda_\nu}$ is 
the symplectic isomorphism described as follows. Consider the map
 $\bar\varphi_{\lambda_\nu}:J^{-1}(\nu)\to T^*(M/G)$
given by
\[
 \bar\varphi_{\lambda_\nu}(\alpha_x)(T_x\pi_{M,G}(v_x))=(\alpha_x-\lambda_\nu(x))(v_x)
\]
for all $\alpha_x\in J^{-1}(\nu)\cap T^*_xM$ and $v_x\in T_xM$. This map is 
invariant with respect to $\phi:G\times J^{-1}(\nu)\to J^{-1}(\nu)$. The 
corresponding quotient map from $J^{-1}(\nu)/G=(T^*M)_\nu$ to $T^*(M/G)$ 
is just $\varphi_{\lambda_\nu}$.

Now, we prove that $\bar\varphi_{\lambda_\nu}$ is equivariant with respect 
to the $\Rr$-actions $\psi_\pi:\Rr\times J^{-1}(\nu)\to J^{-1}(\nu)$ and 
$\psi_{\pi^*_{1,\nu}}:\Rr\times T^*(M/G)\to T^*(M/G)$, that is
\begin{equation}\label{MainThmThesis1}
 \bar\varphi_{\lambda_\nu}\circ (\psi_\pi)_s=(\psi_{\pi^*_{1,\nu}})_s\circ \bar\varphi_{\lambda_\nu},
  \qquad\mbox{for any }s\in\Rr.
\end{equation}
In fact, for all $\alpha_x\in J^{-1}(\nu)\cap T^*_xM$ and $v_x\in T_xM$
\[
 [(\psi_{\pi^*_{1,\nu}})_s\circ \bar\varphi_{\lambda_\nu}](\alpha_x)(T_x\pi_{M,G}(v_x)) =  
  (\alpha_x-\lambda_\nu(x))(v_x)+s\,\pi_{M,G}^*(\widetilde{\eta}_{[x]})(v_x)
\]
where $\widetilde{\eta}=(\pi^*_{1,\nu})^*(dt)$. From \eqref{defTildePi}, we deduce 
\eqref{MainThmThesis1}. 

If $\pi_\nu:J^{-1}(\nu)\to J^{-1}(\nu)/G_\nu$ denotes the quotient map, 
from  \eqref{MainThmThesis1} and since $\varphi_{\lambda_\nu}\circ\pi_\nu=\bar\varphi_{\lambda_\nu}$ and $\pi_\nu$ is equivariant with respect to the principal $\Rr$-actions, 
we have that
\[
 \varphi_{\lambda_\nu}\circ((\psi_\pi)_\nu)_s\circ\pi_\nu 
   = \varphi_{\lambda_\nu}\circ\pi_\nu\circ(\psi_\pi)_s 
   = (\psi_{\pi^*_{1,\nu}})_s\circ\varphi_{\lambda_\nu}\circ\pi_\nu,
\]
$(\psi_\pi)_\nu:\Rr\times (T^*M)_\nu\to(T^*M)_\nu$ being the $\Rr$-action on 
the reduced space $(T^*M)_\nu$ induced by $\psi_\pi$. Thus, using the 
fact that $\pi_\nu$ is surjective, we conclude that $\varphi_{\lambda_\nu}$ 
is a symplectic principal $\Rr$-bundle morphism in the case $G=G_\nu$.

Finally, suppose that $\nu$ is an arbitrary element of $\g^*.$ 
Consider the action $\phi_\nu:G_\nu\times M\to M$ induced by $\phi$.
Its cotangent lift $T^*\phi_\nu:G_\nu\times T^*M\to T^*M$ is a symplectic action which admits   an 
$\Ad^*$-equivariant momentum map $J_\nu:T^*M\to\g_\nu^*$ given by 
\eqref{Jstep3}. 

If $\nu'=\nu_{|\g_\nu}\in\g_\nu^*$ then $\nu'$ is a fixed point of the coadjoint action of $G_\nu$, i.e. $(G_\nu)_{\nu'}=G_\nu.$ Moreover, the $G_\nu$-invariant 
embedding $\bar{\iota}:J^{-1}(\nu)\hookrightarrow J_\nu^{-1}(\nu')$ 
descends to the quotient  and we get 
\begin{equation}\label{emb}
 \iota:J^{-1}(\nu)/G_\nu\hookrightarrow J_\nu^{-1}(\nu')/G_\nu.
\end{equation}

Note that $\bar{\iota}$ is equivariant with respect to the 
$\Rr$-actions $\psi_\pi:\Rr\times J^{-1}(\nu)\to J^{-1}(\nu)$ and 
$\psi_\pi:\Rr\times J_\nu^{-1}(\nu')\to J_\nu^{-1}(\nu')$. Thus, $\iota$ 
is equivariant with respect to the reduced $\Rr$-actions
\[
 (\psi_\pi)_\nu:\Rr\times J^{-1}(\nu)/G_\nu\to J^{-1}(\nu)/G_\nu
\]
and 
\[
 (\psi_\pi)_{\nu'}:\Rr\times J_\nu^{-1}(\nu')/G_\nu\to J_\nu^{-1}(\nu')/G_\nu.
\]
On the other hand, $J^{-1}(\nu)/G_\nu$ (resp.,
$J_\nu^{-1}(\nu')/G_\nu$) is the total 
space of the reduced symplectic principal $\Rr$-bundle $(\mu_\pi)_\nu$ 
(resp., $(\mu_\pi)_{\nu'}$) obtained from $\mu_\pi$ using the 
canonical action of $G$ (resp., $G_\nu$) on $T^*M$. 
Consequently, $\iota$ is a symplectic principal $\Rr$-bundle embedding.

Now, using that $(G_\nu)_{\nu'}=G_\nu$ and the first part of the proof, 
we have a symplectic principal $\Rr$-bundle isomorphism
\[
 \varphi_{\lambda_{\nu'}}:(T^*M)_{\nu'}\to T^*(M/G_\nu)
\]
between the reduced symplectic principal $\Rr$-bundles 
$(\mu_\pi)_{\nu'}:((T^*M)_{\nu'},(\Omega_M)_{\nu'})\to (V^*\pi)_{\nu'}$ 
and 
$\mu_{\pi^*_{1,\nu}}:(T^*(M/G_\nu),\Omega_{M/G_\nu}-B_{\lambda_\nu})\to V^*\pi^*_{1,\nu}$.

Composing $\iota$ with $\varphi_{\lambda_{\nu'}}$, we obtain the required embedding 
$\varphi_{\lambda_\nu}$. 
\end{proof}

Under the same hypotheses as in Theorem \ref{StandardReductionThm},  using  (\ref{r1}) and (\ref{r2}), it follows that the Poisson structures on $T^*(M/G_\nu)$ and $V^*\pi^*_{1,\nu}$ are 
$$\Lambda_{T^*(M/G_\nu)}+ \beta_{\lambda_{\nu}}^v\mbox{ and } \Lambda_{V^*\pi^*_{1,\nu}} + \bar{\beta}_{\lambda_\nu}^v, $$
respectively, where $\bar{\beta}_{\lambda_\nu}$ is the restriction of  ${\beta}_{\lambda_\nu}$  to $V\pi^*_{1,\nu}\times  V\pi^*_{1,\nu}.$ 

On the other hand, if $\varphi^V_{\lambda_{\nu}}:(V^*\pi)_\nu\to V^*\pi^*_{1,\nu}$ is the corresponding embedding between the base spaces  of the  principal $\Rr$-bundles, $[\alpha_x]\in (T^*M)_\nu$ and $[\bar\alpha_x]\in (V^*\pi)_\nu, $ then from 
\eqref{DirectSumDual}, we have that 
\begin{gather*}
  T^*_{\varphi_{\lambda_\nu}[\alpha_x]}(T^*(M/G_\nu)) 
   = \big(T_{[\alpha_x]}(T^*M)_\nu\big)^o 
   \oplus \sharp_{\Lambda_{T^*(M/G_\nu)}+\beta_{\lambda_\nu}^v}^{-1}(T_{[\alpha_x]}(T^*M)_\nu), \\
  T^*_{\varphi_{\lambda_\nu}^V[\bar\alpha_x]}(V^*\pi^*_{1,\nu}) 
   = \big(T_{[\bar\alpha_x]}(V^*\pi)_\nu\big)^o 
   \oplus \big(\sharp_{\Lambda_{V^*\pi^*_{1,\nu}}+\bar\beta_{\lambda_\nu}^v}(T_{[\bar\alpha_x]}(V^*\pi)_\nu)^o\big)^o,
 \end{gather*}
and the corresponding projectors 
\begin{gather*}
 Q^*_{[\alpha_x]}:T^*_{\varphi_{\lambda_\nu}[\alpha_x]}(T^*(M/G_\nu))\to \big(T_{[\alpha_x]}(T^*M)_\nu\big)^o  \\
 q^*_{[\bar\alpha_x]}:T^*_{\varphi_{\lambda_\nu}^V[\bar\alpha_x]}(V^*\pi^*_{1,\nu}) \to \big(T_{[\bar\alpha_x]}(V^*\pi)_\nu\big)^o.
\end{gather*}

Moreover, using Proposition \ref{PoissonEmbedding}, we conclude that
\begin{thm}
Under the same hypotheses as in Theorem \ref{StandardReductionThm}, the reduced Poisson 
structures $\Lambda_\nu$ and $\bar\Lambda_\nu$ on $(T^*M)_\nu$ and $(V^*\pi)_\nu,$ 
respectively,  are given by
\[
\begin{split}
 \Lambda_\nu([\alpha_x])\big( \varphi_{\lambda_\nu}^*\alpha'_1 , \varphi_{\lambda_\nu}^*\alpha'_2\big)
       &=  ( \Lambda_{T^*(M/G_\nu)}+\beta_{\lambda_\nu}^v) (\varphi_{\lambda_\nu}[\alpha_x]) \big(\alpha'_1,\alpha'_2\big)\\[5pt]
       &\ \ -( \Lambda_{T^*(M/G_\nu)}+\beta_{\lambda_\nu}^v) (\varphi_{\lambda_\nu}[\alpha_x]) \big( Q^*_{[\alpha_x]}\alpha'_1 , Q^*_{[\alpha_x]}\alpha'_2 \big)
\end{split}
\]
and
\[
\begin{split}
  \bar\Lambda_\nu([\bar\alpha_x])\big( (\varphi_{\lambda_\nu}^V)^*\sigma'_1 , (\varphi_{\lambda_\nu}^V)^*\sigma'_2 \big)  
       &= ( \Lambda_{V^*\pi^*_{1,\nu}}+\bar\beta_{\lambda_\nu}^v) (\varphi_{\lambda_\nu}^V[\bar{\alpha}_x]) \big(\sigma'_1 ,\sigma'_2 \big)\\[5pt]
       &\ \ -( \Lambda_{V^*\pi^*_{1,\nu}}+\bar\beta_{\lambda_\nu}^v) (\varphi_{\lambda_\nu}^V[\bar{\alpha}_x]) \big( q^*_{[\bar\alpha_x]}\sigma'_1 , q^*_{[\bar\alpha_x]}\sigma'_2 \big).
\end{split}
\]
for all $\alpha'_1, \alpha'_2\in T^*_{\varphi_{\lambda_\nu}[\alpha_x]}(T^*(M/G_\nu))$, 
$\sigma'_1,\sigma'_2\in T^*_{\varphi_{\lambda_\nu}^V[\bar\alpha_x]}(V^*\pi^*_{1,\nu})$ 
and $[\alpha_x]\in(T^*M)_\nu$, $[\bar\alpha_x]\in (V^*\pi)_\nu$.
\end{thm}

\begin{ex}[{\it The bidimensional time-dependent damped harmonic oscillator}] \label{ex4.8}  
(see \cite{CaLuRa} and references therein).
This time-dependent mechanical system involves harmonic oscillators with time-dependent frequency or with time-dependent masses or subject to linear time-dependent damping forces. The configuration space is the manifold $\Rr^2\times \Rr$ fibered on $\Rr$ with respect to the surjective submersion $pr_2: \Rr^2\times \Rr\to \Rr$ and the corresponding restricted phase space of momenta is $V^*pr_2=T^*\Rr^2\times \Rr$. The Hamiltonian function $H:V^*pr_2=T^*\Rr^2\times \Rr\cong (\Rr^2\times \Rr^2)\times \Rr\to \Rr$ is given by 
$$H(q^1,q^2,p_1,p_2,t)=\frac{e^{\sigma(t)}}{2}(p_1^2 + p_2^2) + F(t)((q^1)^2 + (q^2)^2)$$
with $\sigma,F:\Rr\to \Rr$  real $C^\infty$-functions on $\Rr.$ 

We consider the action $\phi:S^1\times (\Rr^2\times \Rr)\to (\Rr^2\times \Rr)$ of $S^1$ on $\Rr^2\times \Rr$  given by 
\[
\phi_\theta(q^1,q^2,t)=(q^1\cos \theta + q^2 \sin \theta, -q^1\sin \theta + q^2\cos \theta, t ), \mbox{ for }\theta\in S^1
\]
which is not free. However, if one restricts this action to $(\Rr^2\setminus\{(0,0)\})\times \Rr\cong (S^1\times \Rr^+)\times \Rr$, $\phi$ is free and proper.  Then, in order to reduce the system,  we consider the second projection $\pi:(S^1\times \Rr^+)\times \Rr\to \Rr$ and the corresponding  symplectic $\Rr$-principal bundle 
\[
\mu_\pi:T^*(S^1\times \Rr^+)\times \Rr^2\to T^*(S^1\times \Rr^+)\times \Rr, \quad(\theta,r,p_\theta,p_r, t,p)\mapsto (\theta,r,p_\theta, p_r,t).
\]
A direct computation proves that $\pi\circ \phi_\theta=\pi$. Therefore,  $T^*\phi:S^1\times (T^*(S^1\times \Rr^+)\times\Rr^2)\to  T^*(S^1\times \Rr^+)\times\Rr^2$ is a canonical action. On the other hand, we have the momentum maps $J:T^*(S^1\times {\Rr}^+)\times \Rr^2\to\Rr$  and $J^{V^*\pi}:T^*(S^1\times {\Rr}^+)\times \Rr\to \Rr$  deduced from (\ref{Jcotangent}) whose explicit expressions are
\[
J(\theta,r,p_\theta,p_r,t,p)=p_\theta,\qquad J^{V^*\pi}(\theta,r,p_\theta,p_r,t)=p_\theta.
\]
If $\nu\in \Rr$, then the corresponding level sets may be expressed as
\[
J^{-1}(\nu)\cong S^1\times T^*\Rr^+\times \Rr^2,\qquad (J^{V^*\pi})^{-1}(\nu)\cong S^1\times T^*\Rr^+\times \Rr
\]
and, since the isotropy subgroup of $S^1$ at $\nu$ is again $S^1$, the reduced spaces are just 
\[
J^{-1}(\nu)/S^1\cong T^*(\Rr^+\times \Rr),\qquad (J^{V^*\pi})^{-1}(\nu)/S^1\cong  T^*\Rr^+\times \Rr.
\]
Finally, the Poisson structure on $(J^{V^*\pi})^{-1}(\nu)/S^1\cong  T^*\Rr^+\times \Rr$ is the one induced by the standard cosympletic structure. 
\end{ex}

\begin{ex}[{\it The time-dependent heavy top}]\label{ex4.9} (see \cite{LeRaSiMa} and references therein). 
This system consists of a rigid body with a fixed point moving in a time-dependent gravitational field. 

The configuration space for this mechanical system is the product manifold $SO(3)\times \Rr$ fibered on $\Rr$ by the second projection $\pi:SO(3)\times \Rr\to \Rr.$ Moreover, the phase space of momenta  $V^*\pi$ may be identified in a natural way with $(SO(3)\times \Rr^3)\times \Rr$ using the left trivialization of $T^*SO(3)$. Under this identification, the Hamiltonian function $H:(SO(3)\times \Rr)\times \Rr^3\to \Rr$ is given by 
\[
 H((A,t), \Pi)=\frac{1}{2} \langle \mathbb{I}^{-1}\Pi,\Pi\rangle  + \langle A^{-1}e_3,\gamma(t)\rangle,
\]
where $\mathbb{I}:\mathfrak{so}(3)\cong \Rr^3\to \mathfrak{so}^*(3)\cong \Rr^3$ is the inertial tensor of the body and $\gamma:\Rr\to\Rr^3$ is the time-dependent gravitational field. 

Now, we consider the closed subgroup of $SO(3)$ 
\[
 K=\left\{\left(\begin{array}{ccc} \cos \theta&\sin \theta&0\\-\sin \theta&\cos \theta&0\\0&0&1\end{array}\right)/\theta\in \Rr\right\}\cong S^1.
\]
If  $\{e_1,e_2,e_3\}$  is the canonical basis of $\mathfrak{so}(3)\cong \Rr^3$, then the Lie algebra associated with $K$ is just $\langle e_3\rangle$. 

In addition, we take the action of $K$ on $T^*(SO(3)\times \Rr)\cong (SO(3)\times \Rr)\times (\Rr^3\times \Rr)$ given by 
\[
 \phi(A_\theta,(A,t),\Pi,p)=((A_\theta A,t),\Pi,p), 
\]
with $A_\theta\in K$, $A\in SO(3)$, $t\in \Rr$ and $(\Pi,p)\in \Rr^3\times \Rr.$ 
This action is free and proper and $\pi$ is invariant with respect to it. 

Let $J:(SO(3)\times \Rr)\times (\Rr^3\times \Rr)\to \Rr$ be the momentum map deduced from \eqref{Jcotangent} whose explicit expression is 
\[
 J((A,t),\Pi,p)=A\Pi\cdot e_3.
\]
Now, for $\nu\in \Rr$, the isotropic subgroup $K_\nu$ is just $S^1$ and the level set $J^{-1}(\nu)$  is  
\[
 \{((A,t),(\Pi,p))\in (SO(3)\times \Rr)\times (\Rr^3\times \Rr)/A\Pi\cdot e_3=\nu\}.
\]

If we apply the cotangent bundle reduction using the principal connection $\lambda:T(SO(3))\to\Rr$ given by $\lambda(A,v)=(Av)\cdot e_3$, we obtain that the reduced symplectic manifold $J^{-1}(\nu)/K_\nu$ is diffeomorphic to $T^*(S^2\times \Rr)\cong T^*S^2\times \Rr^2$. The explicit diffeomorphism  $J^{-1}(\nu)/H_\nu \to T^*S^2\times \Rr^2$ is 
\[
 [(A,t),\Pi,p)]\to  ( (A^{-1}e_3)\times\Pi,(t,p)).
\]
Note that $T^*_{A^{-1}e_3}S^2 \cong \{u\in \Rr^3 / u\cdot (A^{-1}e_3)=0\}$. Moreover, the symplectic $2$-form on $T^*(S^2\times\Rr)$ is $\Omega_{S^2\times \Rr}-\nu\pi_{S^2\times \Rr}^*(\omega_{S^2})$, where $\Omega_{S^2\times \Rr}$ is the canonical symplectic structure on $T^*(S^2\times \Rr)$, $\pi_{S^2\times \Rr}:T^*(S^2\times \Rr)\to S^2\times \Rr$ is the canonical projection and $\omega_{S^2}$ is the symplectic area on $S^2$, i.e. 
\[
 \omega_{S^2}(x)(u,v)=-x\cdot (u\times v),\quad\forall\,x\in S^2,\,u,v\in T_xS^2\subseteq T_x\Rr^3\cong \Rr^3.
\]

On the other hand, on $V^*\pi=(SO(3)\times \Rr)\times \Rr^3,$ the Poisson momentum map $J^{V^*\pi}: (SO(3)\times \Rr)\times \Rr^3\to \Rr$ is given by
\[
 J^{\pi^*V}((A,t),\Pi)=A\Pi\cdot e_3.
\]
The level set $(J^{V^*\pi})^{-1}(\nu)$ (respectively, the reduced space $J^{V^*\pi}(\nu)/K_\nu$) is diffeomorphic to $(SO(3)\times \Rr)\times \Rr^2$ (respectively, to  $T^*S^2\times \Rr$). In order to describe the Poisson structure on this reduced space, we consider the symplectic $2$-form $\bar{\Omega}=\Omega_{S^2} - \nu \pi_{S^2}^*(\omega_{S^2})$ on $T^*S^2\times \Rr$, where $\Omega_{S^2}$ is the canonical symplectic structure of $T^*S^2$ and $\pi_{S^2}:T^*S^2\to S^2$ the canonical projection. Then, the Poisson structure on $T^*S^2\times \Rr$ is just the Poisson structure induced by $pr_1^*(\bar\Omega)$, where $pr_1:T^*S^2\times \Rr\to T^*S^2$ is the canonical projection on the first factor.
\end{ex}

\section{Non-autonomous Hamiltonian reduction}

\subsection{Non-autonomous Hamiltonian systems}\label{section5.1}
In this section, we will extend the example in Section \ref{section2} for a symplectic principal $\Rr$-bundle 
in the presence of a Hamiltonian section. Let $\mu:A\to V$ be a principal $\Rr$-bundle with 
principal action $\psi:\Rr\times A\to A$ and infinitesimal generator $Z_\mu$.

It is well-known (see \cite{GrabGrabUrbAV}) that there exists a one-to-one 
correspondence between the space of the sections of 
$\mu$ and the set $\{F\in C^\infty(A)|\,Z_\mu(F)=1\}$. 
In fact, if $h:V\to A$ is a section of $\mu$, there is a unique function 
$F_h:A\to\Rr$ such that
\begin{equation}\label{DefinitionFh}
 a=\psi\left(F_h(a),h(\mu(a))\right),\qquad\mbox{for any }a\in A.
\end{equation}
 
The following result will be useful in the sequel.
\begin{lem}\label{FhCompatibility}
Let $\mu:A\to V$ be a principal $\Rr$-bundle with principal 
action $\psi:\Rr\times A\to A$ and $h:V\to A$ a section 
of $\mu$. Then
\begin{align}
 F_h(h(v))        &= 0,         \qquad        
    &\mbox{for any }v\in V, \label{FhCirch}\\
 F_h(\psi(s,a))   &= s + F_h(a),\qquad        
    &\mbox{for any }s\in\Rr,\ a\in A. \label{FhWithPrincipalAction}
\end{align}
\end{lem}
\begin{proof}
For any $v\in V$, we have (see \eqref{DefinitionFh})
\[
 h(v) = \psi\big( F_h(h(v)) , h(\mu(h(v))) \big) 
      = \psi(F_h(h(v)),h(v)).
\]
Thus,  since $\psi$ is a free action, we deduce that $F_h(h(v))=0$. 

Moreover, from \eqref{DefinitionFh}, for any $s\in\Rr$ and $a\in A$
\begin{equation}\label{FhCompFormula1}
 \psi(s,a) = \psi\big( F_h(\psi(s,a)) , h(\mu(\psi(s,a))) \big) 
           = \psi\big( F_h(\psi(s,a)) , h(\mu(a)) \big). 
\end{equation}
On the other hand, using again \eqref{DefinitionFh}, it follows that 
\begin{equation}\label{FhCompFormula2}
 \psi(s,a) = \psi\big( s, \psi( F_h(a) , h(\mu(a)) ) \big)
           = \psi\big( s + F_h(a) , h(\mu(a))  \big). 
\end{equation}
Comparing  \eqref{FhCompFormula1} and 
\eqref{FhCompFormula2}, we obtain 
\eqref{FhWithPrincipalAction}.
\end{proof}

Using \eqref{FhWithPrincipalAction}, we deduce that 
 $\psi_s^*(\d F_h)=\d F_h,$ for any $s\in \Rr$ and since $\d F_h(Z_\mu)=1$ it follows that $\d F_h:TA\to \Rr$ is the 
connection $1$-form of a principal connection on the principal $\Rr$-bundle $\mu:A\to V$ 
(see \cite{GrabGrabUrbAV}). 

The horizontal subbundle associated with the principal connection is 
\[
 a\in A\mapsto {H}_a^h=\{X\in T_aA|X(F_h)=0\}\subseteq T_aA
\]
and thus 
\begin{equation}\label{Decom} 
T_aA={ H}_a^h\oplus V_a\mu={H}_a^h\oplus \langle Z_\mu(a) \rangle.
\end{equation}
Note that 
\[
 T_{h(\mu(a))}\mu(T_{\mu(a)}h\circ T_a\mu)(X)=(T_a\mu)(X)
\]
and, moreover, from \eqref{FhCirch}
\[
 \{(T_{\mu(a)}h\circ T_a\mu)(X)\}(F_h)=0.
\]
This implies that 
\[
 (T_{h(\mu(a))}\psi_{F_h(a)})((T_{\mu(a)}h\circ T_a\mu)(X))=X-X(F_h)Z_\mu(a)
\]
and, therefore, the horizontal projector $\hor_a^h:T_aA\to H_a^h$ is given by 
\begin{equation} \label{projector-h}
 \hor_a^h(X)=(T_{h(\mu(a))}\psi_{F_h(a)}\circ T_{\mu(a)}h\circ T_a\mu)(X).
\end{equation}

\begin{defn}
 A \emph{non-autonomous Hamiltonian system} 
$(A,\mu,\Omega,h)$ is a symplectic principal $\Rr$-bundle 
$\mu:(A,\Omega)\to V$ endowed with a section $h:V\to A$ 
of $\mu$, i.e.~a smooth map such that $\mu\circ h=id_V$.

The section $h:V\to A$ is called the \emph{Hamiltonian 
section of the system}.
\end{defn}

In this subsection we will prove that, given a non-autonomous
Hamiltonian system $(A,\mu,\Omega,h)$, the base manifold
$V$ of the principal $\Rr$-bundle $\mu$ may be equipped with a cosymplectic structure.

\begin{prop}\label{HaFhProjectable}
Let $(A,\mu,\Omega,h)$ be a non-autonomous Hamiltonian system.
Then the Hamiltonian vector field $\Ha_{F_h}\in\X(A)$ of the 
function $F_h$ associated to a Hamiltonian section 
$h:V\to A$ is $\mu$-projectable to a vector field $\Ree_h$ on $V$. 
\end{prop}
\begin{proof}
Since the infinitesimal generator $Z_\mu$ is 
locally Hamiltonian, for any $a\in A$, there exists a function
$\tau$ defined on an open neighbourhood $U$ of $a$ such that
the restriction of $Z_\mu$ to $U$ is the Hamiltonian vector field of $\tau$. 
Using the definition of $F_h$, we have that on $U$
\[
 \pbr{\tau,F_h}_A = -\Ha_\tau(F_h) = -Z_\mu(F_h) = -1,
\]
$\pbrEm_A$ being the Poisson bracket on $A$ induced by the
symplectic form $\Omega$. As a consequence,
\[
 \lie_{Z_\mu}\Ha_{F_h} = \br{\Ha_\tau,\Ha_{F_h}}
   = -\Ha_{\pbr{\tau,F_h}_A} = 0.
\]
Thus, the Lie derivative of $\Ha_{F_h}$ with respect to any 
vertical vector field is again vertical. This is a sufficient
(and necessary) condition to ensure the $\mu$-projectability
of $\Ha_{F_h}$.
\end{proof}

The $\mu$-projection $\Ree_h$ of $\Ha_{F_h}$ is a vector field
on $V$, which describes the Hamiltonian dynamics of the 
non-autonomous Hamiltonian system $(A,\mu,\Omega,h)$, 
as we will see in what follows.

\begin{prop}\label{FormulaOmegaOmegah}
Let $(A,\mu,\Omega,h)$ be a non-autonomous Hamiltonian system. 
If $\omega_h\in\Omega^2(V)$ and $\eta_h\in\Omega^1(V)$ are defined by
\begin{equation}\label{DefOmegahAndEtah}
 \omega_h=h^*\Omega,\qquad \eta_h=-h^*(i_{Z_\mu}\Omega),
\end{equation}
then 
\begin{equation}\label{OmegaOmegah}
 \Omega=\mu^*\omega_h-\d F_h\wedge\mu^*\eta_h
\end{equation} 
and 
\begin{equation}\label{muStarEta}
 \mu^*\eta_h=-i_{Z_\mu}\Omega. 
\end{equation}
\end{prop}

\begin{proof}
First of all, we will see that \eqref{muStarEta} holds. It is clear that 
\begin{equation}\label{vertical}
(\mu^*\eta_h)(Z_\mu)=-(i_{Z_\mu}\Omega)(Z_\mu)=0.
\end{equation}
On the other hand, from \eqref{projector-h} and since 
$\mu\circ \psi_s=\mu$ and $\mu\circ h=id$, it follows that 
\[
 (\mu^*\eta_h)(a)(\hor_a^h(X))=-[h^*(i_{Z_\mu}\Omega)](\mu(a))((T_a\mu)(X)),
\]
for $a\in A$ and $X\in T_aA.$ 

Thus, since $\psi$ is a symplectic action, we obtain that 
\[
 \mu^*(\eta_h)(a)(\hor_a^h(X))=-(i_{Z_\mu}\Omega)(a)(\hor_a^h(X)).
\] 
This, using \eqref{Decom} and \eqref{vertical}, proves \eqref{muStarEta}. 

Next, we will see that \eqref{OmegaOmegah} holds. 

From \eqref{muStarEta}, we deduce that 
\begin{equation}\label{5.7.1}
i_{Z_\mu}\Omega=i_{Z_\mu}(\mu^*\omega_h-\d F_h\wedge \mu^*\eta_h).
\end{equation}
On the other hand, using \eqref{projector-h} and \eqref{DefOmegahAndEtah} and the fact 
that $\psi$ is a symplectic action, we have that 
\begin{equation}\label{5.7.2}
 \begin{split}
  \Omega(a)(\hor_a^h(X), \hor_a^h(Y)) &= (\mu^*\omega_h)(a)(X,Y) \\
           &= (\mu^*\omega_h)(a)(\hor_a^h(X), \hor_a^h(Y)) \\
           &= (\mu^*\omega_h-\d F_h\wedge \mu^*\eta_h)(\hor_a^h(X), \hor_a^h(Y))
 \end{split}
\end{equation}
for $a\in A$ and $X,Y\in T_aA.$ 

Therefore, from \eqref{Decom}, \eqref{5.7.1} and \eqref{5.7.2}, we deduce \eqref{OmegaOmegah}. 
\end{proof}

Now, we may prove the following result.
\begin{thm}\label{CosymplStructureOnV}
Let $(A,\mu,\Omega,h)$ be a non-autonomous Hamiltonian 
system with infinitesimal generator $Z_\mu$. 
If 
$\omega_h$ and $\eta_h$ are the $1$-form and the $2$-form respectively, on $V$ defined by 
\eqref{DefOmegahAndEtah}, 
then $(V,\omega_h,\eta_h)$ is a cosymplectic manifold. 
The Reeb vector field of the cosymplectic structure on $V$ 
is just $\Ree_h$.
\end{thm}
\begin{proof}
From \eqref{OmegaOmegah} and \eqref{muStarEta} and since $\Omega$ is closed and ${\mathcal L}_{Z_\mu}\Omega=0$, we deduce that $\omega_h$ and $\eta_h$ are closed.

Now, using \eqref{muStarEta} and Proposition \ref{FormulaOmegaOmegah}, we have that 
\begin{equation}\label{5.10'}
 \eta_h(\Ree_h)=(\mu^*\eta_h)({\mathcal H}_{F_h})=(i_{{\mathcal H}_{F_h}}\Omega)(Z_\mu)=1.
\end{equation}
On the other hand, from \eqref{OmegaOmegah} and Proposition \ref{FormulaOmegaOmegah}, it follows that 
\[
 \mu^*(i_{\Ree_h}\omega_h)=i_{{\mathcal H}_{F_h}}(\Omega + \d F_h\wedge \mu^*\eta_h).
\]
Thus, using \eqref{5.10'}, we obtain that $\mu^*(i_{\Ree_h}\omega_h)=0$ which implies that 
\begin{equation}\label{5.10''}
i_{\Ree_h}\omega_h=0.
\end{equation}
Next, suppose that $\dim A=2n+2.$ Then, from \eqref{5.10''}, we deduce that $\mbox{ rank}(\omega_h)\leq 2n$. 
Therefore, using \eqref{FormulaOmegaOmegah}, it follows that 
\[
 0\not=\Omega^{n+1}=c(\mu^*\omega_h)^n\wedge \d F_h\wedge \mu^*\eta_h,\mbox{ with } c\in \Rr, \;\; c\not=0.
\]
Consequently, the rank of $\mu^*\omega_h$ is $2n$ and we have that 
\begin{equation}\label{5.10'''}
\mbox{ rank }\omega_h=2n.
\end{equation}
Conditions \eqref{5.10'}, \eqref{5.10''} and  \eqref{5.10'''} imply that $\eta_h\wedge \omega_h^n\not=0.$
\end{proof}

The cosymplectic structure $(\omega_h,\eta_h)$ on $V$ 
defined on the base manifold $V$ of a non-autonomous 
Hamiltonian system $(A,\mu,\Omega,h)$ induces a Poisson 
structure $\pbrEm_h$ on $V$. 
On the other hand,  as we know (see Proposition \ref{PoissonInducedByMu}), $V$ 
is equipped with a  Poisson 
structure $\pbrEm_V$ in such a way that $\mu$ is a Poisson map.
The next result shows that the Poisson brackets $\pbrEm_h$  and $\pbrEm_V$  are equal. 

\begin{prop}\label{PoissonCosymplOnV}
Let $(A,\mu,\Omega,h)$ be a non-autonomous Hamiltonian system, $\pbrEm_h$ 
the Poisson bracket on $V$ associated with the cosymplectic structure $(\omega_h,\eta_h)$ 
and $\pbrEm_V$ the Poisson bracket on $V$ induced by the symplectic principal $\Rr$-bundle structure. Then, $\pbrEm_h=\pbrEm_V.$ 
\end{prop}

\begin{proof}
Fix a real $C^\infty$-function $f$ on $V$. It is sufficient to prove that the Hamiltonian vector field $X_f$ on $V$ with respect to the Poisson bracket $\pbrEm_V$ is equal to the Hamiltonian vector field of $f$ with respect to the cosymplectic structure  $(\omega_h,\eta_h).$ Note that, since $\mu$ is Poisson map, it follows that the Hamiltonian vector field ${\mathcal H}_{f\circ \mu}\in {\mathfrak X}(A)$ is $\mu$-projectable and its projection is just $X_f$. Thus, from  \eqref{muStarEta}, we have
\begin{equation}\label{FormulaInPoissonCosymplOnV}
 \begin{split}
  \eta_h(X_{f}) &= \mu^*\eta_h(\Ha_{f\circ\mu}) 
    = -i_{Z_\mu}\Omega(\Ha_{f\circ\mu}) 
              =Z_\mu(f\circ\mu)=0.
 \end{split}
\end{equation}
On the other hand, using that ${\mathcal H}_{F_h}$ is $\mu$-projectable on $\Ree_h$, we deduce that 
\[
 \Ree_h(f)=\Ha_{F_h}(f\circ\mu)=-\d F_h(\Ha_{f\circ\mu}).
\]
Now, from \eqref{OmegaOmegah} and \eqref{FormulaInPoissonCosymplOnV}, 
it follows that 
\[
 \begin{split}
  (i_{X_{f}}\omega_h)({\mu(a)})(T_a\mu(\bar{Y})) 
     &= (\mu^*\omega_h)(a)(\Ha_{f\circ\mu}(a),\bar{Y}) \\
     &= \Omega(a)(\Ha_{f\circ\mu}(a),\bar{Y}) 
         + (\d F_h\wedge \mu^*\eta_h)(a)(\Ha_{f\circ\mu}(a),\bar{Y}) \\
     &= (\d(f\circ\mu))(a)(\bar{Y}) 
         + (\d F_h)(a)(\Ha_{f\circ\mu}(a))(\mu^*\eta_h)(a)(\bar{Y}) \\
     &= (\d f-\Ree_h(f)\eta_h )({\mu(a)})(T_a\mu(\bar{Y})),
 \end{split}
\]
for all $\bar{Y}\in T_aA$, with $a\in A$. Therefore, 
$$i_{X_f}\omega_h=df-\Ree(f)\eta_h.$$
This ends the proof of the result. 
\end{proof}

In what follows, we will prove that the integral curves of the vector field $\Ree_h$ satisfy 
local equations which are just the Hamilton equations. 
For this purpose, we will use canonical coordinates on the symplectic principal $\Rr$-bundle $\mu:(A,\Omega)\to M$ 
(see Theorem \ref{locExpr}). 

Let $(t,p,q^i,p_i)$ be canonical coordinates on $A$. 
Suppose that the local expression of the Hamiltonian 
section $h:V\to A$ is 
\[
 h(t,q^i,p_i)=(t,-H(t,q^j,p_j),q^i,p_i),
\]
where $H$ is a local function on $V$. Then, 
$F_h:A\to\Rr$ may be described locally by
\[
 F_h(t,p,q^i,p_i)=p+H(t,q^i,p_i).
\]
Since $(t,p,q^i,p_i)$ are Darboux coordinates for the symplectic
form $\Omega$ on $A$, we have that 
\[
 \begin{split}
  \Ha_{F_h} &= \pd{}{t} - \pd{H}{t}\pd{}{p} 
             + \pd{H}{p_i}\pd{}{q^i} - \pd{H}{q^i}\pd{}{p_i}, \\
   \Ree_h   &= \pd{}{t} + \pd{H}{p_i}\pd{}{q^i} 
             - \pd{H}{q^i}\pd{}{p_i}.
 \end{split}
\]
Finally, the cosymplectic structure $(\omega_h,\eta_h)$ on $V$ is 
locally described by
\[
 \omega_h = -\pd{H}{q^i}\d t\wedge\d q^i -\pd{H}{p_i}\d t\wedge\d p_i + \d q^i\wedge\d p_i,\qquad \eta_h = \d t.
\]

Thus, a curve on $V$ with local expression
\[
 t\mapsto (t,q^i(t),p_i(t))
\]
is an integral curve of $\Ree_h$ if and only if it satisfies the Hamilton equations 
$$ \frac{dq^i}{dt}=\frac{\partial H}{\partial p_i},\;\;\;  \frac{dp_i}{dt}=-\frac{\partial H}{\partial q^i}.$$
Therefore, in the particular case when $\mu$ is the standard  principal $\Rr$-bundle associated with a fibration $\pi:M\to \Rr,$ we recover the results in Section \ref{section2}.

\subsection{Non-autonomous reduction Theorem}

In this subsection we will obtain a reduction theorem 
for non-autonomous Hamiltonian systems.

Let $\mu:(A,\Omega)\to V$ be a symplectic principal $\Rr$-bundle 
with infinitesimal generator $Z_\mu$ and 
$\phi:G\times A\to A$ a canonical action of a Lie group 
$G$ on the symplectic principal $\Rr$-bundle $\mu:(A,\Omega)\to V$. 
Denote by $\phi^V:G\times V\to V$ 
the corresponding action on $V$.

Now suppose that $h:V\to A$ is a Hamiltonian section 
of $\mu$.
\begin{defn}
The Hamiltonian section $h$ is said to be 
\emph{$G$-equivariant} if $h$ is equivariant with respect to
the actions $\phi$ and $\phi^V$, that is,
\begin{equation}\label{hGequiv}
 h\circ\phi^V_g=\phi_g\circ h, \qquad\mbox{for }g\in G.
\end{equation}
\end{defn}

Note that, if $h:V\to A$ is a Hamiltonian section of a principal
$\Rr$-bundle, then, from 
\eqref{muEquivariant} and \eqref{DefinitionFh}, we have that
\begin{equation}\label{Formula1Ginvariance}
 \phi_g(a)=\psi\big( F_h(\phi_g(a)) , h(\phi^V_g(\mu(a))) \big),
\end{equation}
for any $a\in A$ and $g\in G$. On the other hand, applying $\phi_g$ 
to the two sides of \eqref{DefinitionFh} and using 
\eqref{PhiPsiCompatible}, we obtain 
\begin{equation}\label{Formula2Ginvariance}
 \phi_g(a)=\psi\big( F_h(a) , \phi_g(h(\mu(a))) \big).
\end{equation}
Comparing \eqref{Formula1Ginvariance} and 
\eqref{Formula2Ginvariance}, one may deduce that a 
Hamiltonian section $h:V\to A$ is $G$-equivariant 
if and only if the corresponding function $F_h$ is $G$-invariant, 
i.e.~$F_h\circ\phi_g=F_h$ for any $g\in G$.

\begin{prop}\label{CosymplecticAction}
If $h:V\to A$ is a $G$-equivariant Hamiltonian section, the
induced action $\phi^V:G\times V\to V$ is a cosymplectic 
action with respect to the cosymplectic structure $(\omega_h,\eta_h)$
on $V$ defined by $h$.

Moreover, if $J:A\to\g^*$ is a momentum map with 
respect to the action $\phi$, then the induced momentum map 
$J^V:V\to\g^*$ is such that $\Ree_h(J^V_\xi)=0$, for any $\xi\in\g$.
\end{prop}
\begin{proof}

From the equivariance of $h$ and the $G$-invariance of $\Omega$ and $i_{Z_\mu}\Omega$, we have that $\phi^V_g$ preserves the forms $\omega_h$ and $\eta_h$, for any $g\in G$.
Thus, $\phi^V$ is a cosymplectic action. 
Moreover, for any $\xi\in\g$ we have (see \eqref{JVandJ})
\[
 \Ree_h(J^V_\xi) = \Ha_{F_h}(J^V_\xi\circ\mu) = \Ha_{F_h}(J_\xi) = -\Ha_{J_\xi}(F_h) = -\xi_A(F_h)=0,
\]
$\xi_A$ being the infinitesimal generator of $\phi$ defined 
by $\xi$. The last equality follows from the $G$-invariance 
of $F_h$.
\end{proof}

Now, we may reduce the non-autonomous Hamiltonian system. 
Let $(A,\mu,\Omega,h)$ be a non-autonomous Hamiltonian system 
equipped with a canonical action $\phi:G\times A\to A$ 
of a Lie group $G$ on the manifold $A$ with an $\Ad^*$-equivariant 
momentum map $J:A\to\g^*$. Suppose that 
the induced action $\phi^V:G\times V\to V$ on $V$ is free 
and proper. Let $\nu$ be an element of $\g^*$. Then we induce 
a free and proper action 
$\phi^V:G\times (J^V)^{-1}(\nu)\to (J^V)^{-1}(\nu)$ of $G$ on 
$(J^V)^{-1}(\nu)$ and we may apply Albert's Theorem (see Theorem 
\ref{CosymplecticReduction} in Appendix \ref{MWRed}). Thus, we 
may reduce the cosymplectic manifold $(V,\omega_h,\eta_h)$ 
for obtaining a reduced cosymplectic manifold 
$(V_\nu,(\omega_h)_\nu,(\eta_h)_\nu)$, where $V_\nu$
is the quotient manifold $(J^V)^{-1}(\nu)/G_\nu$ and 
$(\omega_h)_\nu$, $(\eta_h)_\nu$ are the $2$-form and $1$-form 
on $V_\nu$ characterized by
\begin{equation}\label{ConditionsReducedCosymplStruct}
 (\pi^V_\nu)^*(\omega_h)_\nu=(i^V_\nu)^*\omega_h,\qquad 
 (\pi^V_\nu)^*(\eta_h)_\nu=(i^V_\nu)^*\eta_h 
\end{equation}
$\pi^V_\nu:(J^V)^{-1}(\nu)\to V_\nu$ and 
$i^V_\nu:(J^V)^{-1}(\nu)\hookrightarrow V$ being the 
canonical projection and the canonical inclusion, 
respectively.

On the other hand, from Theorem \ref{AVreductionTheorem}, 
we obtain a symplectic principal $\Rr$-bundle $\mu_\nu:(A_\nu,\Omega_\nu)\to V_\nu$
with infinitesimal generator $Z_{\mu_\nu}$. We recall that the restriction 
of $Z_\mu$ to $J^{-1}(\nu)$ is tangent to $J^{-1}(\nu)$ and that 
${Z_\mu}_{|J^{-1}(\nu)}$ is $\pi_\nu$-projectable on $Z_{\mu_\nu}$. 
Moreover,  using \eqref{JVandJ} and \eqref{hGequiv}, we have that $h$ restricts to a $G_\nu$-invariant map $h:(J^V)^{-1}(\nu)\to J^{-1}(\nu)$. Therefore, $h$ induces a smooth map $h_\nu:V_\nu\to A_\nu$ characterized by 
\begin{equation}\label{reduced_h}
 h_\nu\circ \pi^V_\nu = \pi_\nu\circ h. 
\end{equation}
The function $h_\nu$ is a section of $\mu_\nu$, so $h_\nu$ is a
Hamiltonian section of the symplectic principal $\Rr$-bundle $\mu_\nu:A_\nu\to V_\nu$ 
and $(A_\nu,\mu_\nu,\Omega_\nu,h_\nu)$ is a non-autonomous 
Hamiltonian system. 

A direct computation, using \eqref{muAndPi}, \eqref{ReducedAction}, 
\eqref{DefinitionFh} and \eqref{reduced_h}, shows that the function 
$F_{h_\nu}$ on $A_\nu=J^{-1}(\nu)/G_\nu$ is characterized by the following
condition
\begin{equation}\label{ReducedFh}
 F_{h_\nu}\circ\pi_\nu={F_h}_{|J^{-1}(\nu)}.
\end{equation}
Thus, $F_{h_\nu}$ may be obtained from $F_h$ by passing to quotient.

Moreover, from Theorem \ref{CosymplStructureOnV}, 
$(V_\nu,(\omega_\nu)_{h_\nu}, (\eta_\nu)_{h_\nu})$ is a 
cosymplectic manifold whose structure is given by
\begin{equation}\label{redCosForms}
 (\omega_\nu)_{h_\nu} = h_\nu^*\Omega_\nu,
  \qquad (\eta_\nu)_{h_\nu} = -h_\nu^*(i_{Z_{\mu_\nu}}\Omega_\nu).
\end{equation}

\begin{thm}
Let $(A,\mu,\Omega,h)$ be a non-autonomous Hamiltonian 
system and $\phi:G\times A\to A$ be a canonical action 
of $G$ on $A$ such that the induced action on $V$ is free and proper. 
Suppose that $J:A\to\g^*$ is an $\Ad^*$-equivariant 
momentum map. If $h$ is $G$-equivariant, then, for any $\nu\in\g^*$, the 
cosymplectic structure $((\omega_\nu)_{h_\nu}, (\eta_\nu)_{h_\nu})$ 
on $V_\nu$ induced by the reduced non-autonomous Hamiltonian system 
$(A_\nu,\mu_\nu,\Omega_\nu,h_\nu)$ is the one deduced from Albert's 
reduction of the cosymplectic structure $(\omega_h,\eta_h)$ on $V$. 
In other words, 
\begin{equation}\label{RedCosymplStruct}
 (\omega_\nu)_{h_\nu} = (\omega_h)_\nu,\qquad
   (\eta_\nu)_{h_\nu} = (\eta_h)_\nu.
\end{equation}
In particular, the dynamics $\Ree_{h_\nu}$ of the reduced 
non-autonomous Hamiltonian system is just the 
$\pi^V_\nu$-projection of the restriction to $(J^V)^{-1}(\nu)$ of 
the dynamics $\Ree_h$ of $(A,\mu,\Omega,h)$. More precisely, 
\begin{equation}\label{ReducedDynamic}
 \Ree_{h_\nu}(\pi^V_\nu(v))=T_v\pi^V_\nu(\Ree_h(v)),
  \qquad\mbox{for any }v\in(J^V)^{-1}(\nu).
\end{equation}
\end{thm}
\begin{proof}
In order to show \eqref{RedCosymplStruct}, we will prove that the $2$-form 
$(\omega_\nu)_{h_\nu}$ and the $1$-form $(\eta_\nu)_{h_\nu}$ satisfy  
condition \eqref{ConditionsReducedCosymplStruct}. In fact, if 
$v\in(J^V)^{-1}(\nu)$ and $X,Y\in T_v((J^V)^{-1}(\nu))$, then, 
using \eqref{reduced_h} and \eqref{redCosForms}, we have that 
\[
 \begin{split}
  \left( (\pi^V_\nu)^*(\omega_\nu)_{h_\nu} \right)(v)&(X,Y) 
   = (\omega_\nu)_{h_\nu}(\pi^V_\nu(v))
      (T_v\pi^V_\nu(X),T_v\pi^V_\nu(Y)) \\
  &= \Omega_\nu(\pi_\nu(h(v))) 
     \left( T_{h(v)}\pi_\nu(T_vh(X)), 
            T_{h(v)}\pi_\nu(T_vh(Y)) \right) \\
  &= (i_\nu^*\Omega)(h(v)) \left( T_vh(X), T_vh(Y) \right) 
   = (i_\nu^*\omega_h)(v)(X,Y).
 \end{split}
\]
Thus, $(\pi^V_\nu)^*(\omega_\nu)_{h_\nu}=i_\nu^*\omega_h$. 
Moreover, using again \eqref{reduced_h} and \eqref{redCosForms}, we deduce that
\[
 \begin{split}
  \big( (\pi^V_\nu)^*(\eta_\nu)_{h_\nu} \big)(v)(X) 
  &= - \left( h_\nu^*(i_{Z_{\mu_\nu}}\omega_\nu) \right)(\pi^V_\nu(v)) 
       \left( T_v\pi^V_\nu(X) \right) \\
  &= - \left( i_{Z_{\mu_\nu}}\Omega_\nu \right)(\pi_\nu(h(v)))
       \left( T_{h(v)}\pi_\nu(T_vh(X)) \right) \\
  &= - \left( i_\nu^*(i_{Z_\mu}\omega) \right)(h(v)) (T_vh(X)) 
   = (i_\nu^*\eta_h)(v) (X).
 \end{split}
\]
Therefore, $(\pi^V_\nu)^*(\eta_\nu)_{h_\nu}=i_\nu^*\eta_h$.
\end{proof}

Note that another way to obtain \eqref{ReducedDynamic} is 
to use \eqref{ReducedFh}. In fact, since $\Ha_{F_h}$ (respectively, 
$\Ha_{F_{h_\mu}}$) is $\mu$-projectable (respectively, $\mu_\nu$-projectable) 
on $\Ree_h$ (respectively, $\Ree_{h_\nu}$), using 
\eqref{HamiltonianRelatedPoisson}, we have that
\[
 \begin{split}
  \Ree_{h_\nu}(\pi^V_\nu(v)) 
  &= T_{\pi_\nu(a)}\mu_\nu(\Ha_{F_{h_\mu}}(\pi_\nu(a))) 
   = T_{\pi_\nu(a)}\mu_\nu \left( T_a\pi_\nu(\Ha_{F_h}(a)) \right) \\
  &= T_{\mu(a)}\pi^V_\nu \left( T_a\mu(\Ha_{F_h}(a))\right)
   = T_{\mu(a)}\pi^V_\nu ( \Ree_h(v) ),
 \end{split}
\]
for any $v=\mu(a)\in(J^V)^{-1}(\nu)$.

\begin{ex}[{\it The bidimensional time-dependent damped harmonic oscillator}]  
Using the same notation as in Example \ref{ex4.8}, we have that the extended Hamiltonian function $F_h:T^*(S^1\times \Rr^+)\times \Rr^2\to \Rr$ is given by 
\[
 F_h(\theta,r,p_\theta,p_r,t,p)=\frac{e^{\sigma(t)}}{2}(p_r^2+\frac{1}{r^2}p_\theta^2) + F(t)r^2 + p.
\]
Since $F_h$ is $T^*\phi$-invariant, we may reduce the non-autonomous Hamiltonian system $(T^*(S^1\times\Rr^+)\times\Rr^2,\mu_\pi,\Omega,h)$ at $\nu\in\Rr$. The reduced homogeneous Hamiltonian system $F_{h_\nu}:T^*(\Rr^+\times\Rr)\to\Rr$ is given by
\[
 F_{h_\nu}((r,t),(p_r,p))=\frac{e^{\sigma(t)}}{2}\left(p_r^2 + \frac{\nu^2}{r^2}\right) + F(t)r^2 + p.
\]
Finally, the reduced dynamics is described by the cosymplectic structure $((\omega_{h})_\nu,(\eta_{h})_\nu)$ and the vector field ${\mathcal R}_{h_\nu}$  on $T^*\Rr^+\times \Rr$
\begin{gather*}
\omega_{h_\nu}=dr\wedge dp_r + \left(2F(t)r-e^{\sigma(t)}\frac{\nu^2}{r^3}\right) dr\wedge dt + e^{\sigma(t)}p_r dp_r\wedge dt,\quad \eta_{h_\nu}=dt, \\
{\mathcal R}_{h_\nu}=\displaystyle\frac{\partial }{\partial t} + {e^{\sigma(t)}}p_r\frac{\partial }{\partial r}+\left(e^{\sigma(t)}\frac{\nu^2}{r^3} - 2F(t)r\right)\frac{\partial}{\partial p_r}.
\end{gather*}

\end{ex}

\begin{ex}[{\it The time-dependent heavy top}]
Using the same notation as in Example \ref{ex4.9}, we have that the extended Hamiltonian homogeneous function $F_h:(SO(3)\times \Rr)\times (\Rr^3\times \Rr)\to\Rr$ associated with this system is  
\[
 F_h((A,t),\Pi,p)=\frac{1}{2}\langle \mathbb{I}^{-1}\Pi,\Pi\rangle+ \langle A^{-1}e_3,\gamma(t) \rangle + p.
\]

When one applies the reduction process at the level $\nu=0$, we obtain the reduced Hamiltonian homogenous function $F_{h_0}:T^*S^2\times \Rr^2\to \Rr$ which is the restriction to $T^*S^2\times \Rr^2$ of $\widetilde{F}_{h_0}:\Rr^3\times \Rr^3\times \Rr^2\to \Rr$ 
\[
 \widetilde{F}_{h_0}((q,p_q,t,p))=\frac{1}{2}\langle \mathbb{I}^{-1}(p_q\times q),(p_q\times q)\rangle + \langle q,\gamma(t) \rangle + p.
\]

In fact, the equations defining to $T^*S^2$ as a submanifold of $T^*\Rr^3\cong \Rr^3\times \Rr^3$ are $\|q\|^2-1=0$ and $q\cdot p_q=0$. So, any extension  of $F_{h_0}$ has the form 
\[
 \bar{F}_{h_0}(q,p_q,t,p)=\widetilde{F}_{h_0}(q,p_q,t,p) + \lambda (q\cdot p_q) + \mu(\|q\|^2-1),
\]
where $\lambda$ and $\mu$ are the Lagrange multipliers which we must determine. Then, the
Hamilton equations for this Hamiltonian function with initial condition on $T^*S^2\times \Rr$ are 
\begin{equation*}
 \left\{\begin{array}{rcl}
  \dot{q} &=& q\times \mathbb{I}^{-1}(p_q\times q)+\lambda q \\[3pt]
  \dot{p_q} &=& p_q\times\mathbb{I}^{-1}(p_q\times q) - \gamma(t) - \lambda p_q -2\mu q \\[3pt] 
  \dot{t}&=&1
\end{array}\right.
\end{equation*}
with $(q,p_q)\in T^*S^2.$ Since $q\cdot \dot{q}=0$ and $p_q\cdot \dot{q} + \dot{p_q}\cdot q=0$, then 
\begin{equation*}\label{redeq}
 \left\{\begin{array}{rcl}
  \dot{q} &=& q\times \mathbb{I}^{-1}(p_q\times q) \\[3pt]
  \dot{p_q} &=& p_q \times \mathbb{I}^{-1}(p_q\times q) - \gamma(t) + \langle q, \gamma(t) \rangle q \\[3pt] 
  \dot{t}&=&1
\end{array}\right.
\end{equation*}
The solutions of these equations are just the integral curves of the Reeb vector field associated with the cosymplectic manifold $T^*S^2\times\Rr$ equipped with the cosymplectic structure $(\Omega_{S^2}+\d H_0\wedge \d t,\d t)$ on $T^*S^2\times\Rr$, where $H_0:T^*S^2\times\Rr\to\Rr$ is given by
\[
H_0(q,p_q,t)=\frac{1}{2}\langle{\Bbb I}^{-1}(p_q\times q),p_q\times q\rangle + \langle q,\gamma(t)\rangle.
\]

Note that if $\mathbb{I}=Id$ and $\gamma(t)=0$ then the corresponding reduced Hamilton equations are 
\[
 \dot{q}=p_q,\qquad\dot{p_q}=-\|\dot{q}\|^2q,\qquad\dot{t}=1
\]
or equivalently $\mbox{\it \"q}=-\|\dot{q}\|^2q$ and $\dot{t}=1.$ Therefore, the geodesics of the standard metric of $S^2\times \Rr$ are just the solutions of the previous equations. 
\end{ex}

\section{Another Example: the frame-independent formulation of the analytical mechanics in a Newtonian space-time}\label{Section6}

The \emph{Newtonian space-time} is a system $(E,\tau,g)$ where $E$ is an affine space modelled over the $n$-dimensional vector space $V$, $\tau$ is a non-zero element of $V^*$ and $g:V_0\to V_0^*$ is a scalar product on $V_0=\ker\tau$. Let $V_1$ be the affine subspace of $V$ defined by the equation $\tau(v)=1$. An element of $V_1$ may be interpreted as the family of inertial observers that move in the space-time with the constant velocity $u$ (see \cite{GrUr}). We will denote by $i:V_0\to V$ inclusion and by $g'=i\circ g^{-1}\circ i^*:V^*\to V$ the contravariant tensor on $V$ defined by $g$.

If $u$ is a fixed inertial frame, the homogeneous Hamiltonian function on $T^*E\simeq E\times V^*$ is given by
\[ 
 H_u(x,\alpha)=\alpha(u)+\frac{1}{2m}\alpha(g'(\alpha))+\varphi(x), \qquad\mbox{for any }x\in E,\ \alpha\in V^*,
\]
where $\varphi:E\to\Rr$ is the potential. The aim of this section is to describe a frame-independent formulation of the dynamics and describe how a reduction procedure may be applied in the symplectic principal $\Rr$-bundle setting. 

Consider the following equivalence relations on $V_1\times E\times V^*$ and $V_1\times E\times V_0^*$, respectively
\[
 \begin{split}
  (u,x,\alpha)\sim(u',x',\alpha') & \Leftrightarrow x=x'\mbox{ and }\alpha'=\alpha+m\sigma(u,u') \\
  (u,x,\bar\alpha)\sim_0(u',x',\bar\alpha') & \Leftrightarrow x=x'\mbox{ and }\bar{\alpha}'=\bar{\alpha}+mg(u-u'), 
 \end{split}
\]
for any $(u,x,\alpha),(u',x',\alpha')\in V_1\times E\times V^*$ and $(u,x,\bar\alpha),(u',x',\bar\alpha')\in V_1\times E\times V_0^*$, where $\sigma:V_1\times V_1\to V^*$ is the map defined by
\[
 \sigma(u,u')(v)=g(u-u')\left(v-\tau(v)\frac{u+u'}{2}\right),\qquad v\in V.
\]
Then, one may easily prove that the quotient space $P=(V_1\times E\times V^*)/\sim$ is an affine bundle over $E$ modelled over the vector bundle $E\times V^*\to E$.
Moreover, if $u\in V_1$ is fixed, there exists a unique symplectic form $\Omega$ on $P$ such that the map $\Theta_u:T^*E\cong E\times V^*\to P$ given by $\Theta_u(x,\gamma)=[(u,x,\gamma)]$ is a symplectomorphism, where $E\times V^*\simeq T^*E$ is equipped with the canonical symplectic $2$-form. If $u'\in V_1$ then, since  $\Theta_u^{-1}\circ\Theta_{u'}:T^*E\to T^*E$ is just the translation by the constant $1$-form $\sigma(u',u)$, $\Omega$ doesn't depend on the choose of $u$. On the other hand, we may consider the action $\psi:\Rr\times P\to P$ and the projection $\mu:P\to P_0$ given by
\[
 \psi(s,[u,x,\alpha])=[u,x,\alpha+s\tau],\qquad \mu[u,x,\alpha]=[u,x,\alpha_{|V_0}], 
\]  
where $P_0$ is the quotient space $(V_1\times E\times V_0^*)/\sim_0$. Then, $\mu:(P,\Omega)\to P_0$ is a symplectic principal $\Rr$-bundle  and the principal action of $\Rr$ on $P$ is just $\psi$. 

Consider the following Hamiltonian section $h:P_0\to P$
\[
 h[u,x,\bar\alpha]=\left[u,x,\bar\alpha\circ i_u-\left(\frac{1}{2m}\bar\alpha(g^{-1}(\bar\alpha))+\varphi(x)\right)\tau\right],
\]
where $i_u:V\to V_0$ is the projection $v\mapsto v-\tau(v)u$. Note that the corresponding function $F_h:P\to\Rr$ is just the homogeneous Hamiltonian function, that is $F_h[u,x,\alpha]=H_u(x,\alpha)$. Thus, $(P,\mu,\Omega,h)$ is a non-autonomous Hamiltonian system (\emph{Frame-independent dynamical system in a Newtonian space-time}). 

Now, we will introduce a symmetry in the system: let $G$ a subgroup of the group of the affine transformation of $A$. Suppose that for any $f_L\in G$, where $L:V\to V$ is the corresponding linear map, we have that
\[
 L^*\tau=\tau,\ \ L_{|V_0}\textrm{ preserves $g$ and }\varphi\circ f_L=\varphi.
\]
Moreover, we will suppose that ${\mathfrak g}\subset Aff(E,V_0).$ Then, we may consider the action $\phi:G\times P\to P$ defined by
\[
 (f_L,[u,x,\alpha])\mapsto [Lu,f_L(x), (L^{-1})^*\alpha].
\]
A straightforward computation shows that $\phi$ is a canonical action on the symplectic principal $\Rr$-bundle $\mu$. Finally, we will suppose that the induced action $\phi^{P_0}:G\times P_0\to P_0$ is free and proper. If not, one may restrict to a subset of $E$ (supposed open) and repeat the proofs. For any reference frame $u\in V_1$, one may consider the momentum map $J_u:P\to\g^*$ defined by
\[
 J_u([w,x,\alpha])=\left(\alpha-m\sigma(u,w)\right)(\xi_E(x)),\qquad\mbox{for any }[w,x,\alpha]\in P\mbox{ and } \xi\in\g,
\]
where $\xi_E\in\X(E)$ is the infinitesimal generator of the natural action of $G$ on $E$ and we are identifying $T_xE\simeq V$. 
Then, one obtain the following result
\begin{thm}
Under the previous hypotheses, if $\nu\in\g^*$ and $u\in V_1$ are fixed, a reduced symplectic principal $\Rr$-bundle $\mu_\nu:(P_\nu,\Omega_\nu)\to (P_0)_\nu$ and a reduced Hamiltonian section $h_\nu:(P_0)_\nu\to P_\nu$ are given, where 
\[
 P_\nu=J_u^{-1}(\nu)/G_\nu,\qquad (P_0)_\nu=(J^{P_0}_u)^{-1}(\nu)/G_\nu.
\]
\end{thm}
Here, $J_u^{P_0}:P_0\to {\mathfrak g}^*$ is the momentum map for the Poisson action of $G$ on $P_0=(V_1\times E\times V_0^*)/\sim_0.$

\section{Conclusions and future work}

A reduction process for a symplectic principal $\Rr$-bundles has been described in this paper. 
In particular, we discuss the case of the standard symplectic  principal $\Rr$-bundle associated with a fibration over the real line.  Finally, we consider the reduction of a non-autonomous Hamiltonian section on a  symplectic principal $\Rr$-bundle. 
 In order to do this, we have obtained previously  a cosymplectic structure 
on the base space of the principal $\Rr$-bundle.

Along the paper, we assume the regularity of the canonical action on the symplectic principal $\mathbb{R}$-bundle. Then, one may ask if a similar construction holds if we relax this assumption  in order to include other examples (see, for instance, Example \ref{ex4.8}). We expect that well-known methods on singular reduction (see \cite{ACG, BaLe, OrRa98, OrRa, PeRoSo}) could be applied in the time-dependent setting. We remark that this reduction process could be not functorial, in the sense that a more general object (as, for instance, a suitable map between stratified spaces) could be needed.

Moreover, it would be interesting to develop the procedure of \emph{the reconstruction} in the 
symplectic principal $\Rr$-bundle framework. The aim is to obtain the dynamics of a symmetric 
non-autonomous Hamiltonian system on a symplectic principal $\Rr$-bundle from the reduced dynamics. 
Classical techniques on symplectic reconstruction (see \cite{MaMoRa}) could be used.

On the other hand, suppose that a canonical action of a connected Lie group $G$ on a symplectic principal $\Rr$-bundle is given and 
that $G$ has a closed normal subgroup $H$. A goal could be to realize the reduced principal $\Rr$-bundle 
in a two step procedure: first reducing by $H$ and then by an appropriate Lie group which is related with the quotient group $G/H$. It would be interesting to discuss this procedure
 which is called \emph{reduction by stages} (for reduction by stages in the symplectic framework, see \cite{MarsEtAl}).  

\appendix
\section{Poisson reduction theorems}\label{MWRed}
In this Appendix, we recall some well-known results about 
Poisson reduction in presence of a  momentum map (for more details, see \cite{AbrMarsd,Alb,Ku, LM, MarsEtAl,MarsRat,MarsWein,OrRa}).

Let $\phi:G\times M\to M$ be an action of a Lie group $G$ on a 
Poisson manifold $(M,\pbrEm)$. The action $\phi$ is said to be
a \emph{Poisson action} if $\phi_g:M\to M$ is a Poisson map 
for any $g\in G$. In such a case, a smooth map 
$J:M\to\g^*$ from $M$ to the dual space $\g^*$ of the Lie algebra 
$\g$ of $G$ is said to be a \emph{momentum map} if the infinitesimal generator
$\xi_M$ of the action associated with any $\xi\in\g$ is the 
Hamiltonian vector field of the function $J_\xi:M\to\Rr$ defined 
by the natural pointwise pairing.
Moreover, $J$ is said to be \emph{$\Ad^*$-equivariant} if it is 
equivariant with respect to the action $\phi$ and to the 
coadjoint action $\Ad^*:G\times\g^*\to\g^*$, i.e.
\[
 J(\phi_g(x))=\Ad^*_{g^{-1}}(J(x)),\qquad\mbox{for any }x\in M.
\]

Note that, if $\phi:G\times M\to M$ is a free and proper 
action and $\nu$ is an element of $\g^*$, then $\nu$ is a 
regular value of $J$ (see for example \cite{MarsEtAl}, pag.~8-9). 
Therefore, $J^{-1}(\nu)$ is a closed submanifold of $M$. 
Moreover, if $G_\nu$ denotes the isotropy group of $\nu$ with 
respect to the coadjoint action, i.e.
\[
 G_\nu=\{g\in G:\,Ad^*_{g}\nu=\nu\},
\]
then $\phi$ induces a free and proper action
\[
 \phi:G_\nu\times J^{-1}(\nu)\to J^{-1}(\nu)
\]
of $G_\nu$ on the submanifold $J^{-1}(\nu)$. 

In addition, we have the following result
\begin{thm}[\bf{Poisson reduction Theorem}, \cite{MarsRat}]\label{PoissonReduction}
 Let $\phi:G\times M\to M$ be a free and proper Poisson 
action of a Lie group $G$ on a Poisson manifold $(M,\pbrEm)$. 
If $J:M\to\g^*$ is an $\Ad^*$-equivariant momentum map 
associated with $\phi$ and $\nu\in\g^*$, then the reduced space 
$M_\nu=J^{-1}(\nu)/G_\nu$ is a Poisson manifold with 
Poisson bracket $\pbrEm_\nu$ characterized by
\begin{equation}\label{ReducedPoissonBracket}
 \pbr{\rho_\nu,\tau_\nu}_\nu(\pi_\nu(x)) = 
 \pbr{\rho,\tau}(x),
 \qquad\mbox{for any }\rho_\nu,\tau_\nu\in C^\infty(M_\nu),
\end{equation}
where $\pi_\nu:J^{-1}(\nu)\to M_\nu$ is the canonical projection and 
$\rho,\tau\in C^\infty(M)$ are arbitrary $G$-invariant extensions of 
$\rho_\nu\circ\pi_\nu$ and $\tau_\nu\circ\pi_\nu,$ respectively.
\end{thm}
Note that, if $\rho$ is a $G$-invariant function on 
$M$ and $\rho_\nu$ is the function on $M_\nu$ such that 
$\rho_\nu\circ\pi_\nu=\rho_{|J^{-1}(\nu)}$, then the restriction to $J^{-1}(\nu)$ 
of $\Ha_\rho$ is tangent to $J^{-1}(\nu)$ and
\begin{equation}\label{HamiltonianRelatedPoisson}
 T_x\pi_\nu(\Ha_\rho(x))=\Ha_{\rho_\nu}(\pi_\nu(x)),
  \qquad\mbox{for all }x\in J^{-1}(\nu).
\end{equation}

The symplectic version of the Poisson reduction Theorem 
is the well-known Marsden-Weinstein reduction theorem. 
In this case, we will assume that the Poisson action is symplectic.

Let $\phi:G\times M\to M$ be an action of a Lie group $G$ on a symplectic 
manifold $(M,\Omega)$. The action $\phi$ is said to be 
\emph{symplectic} if $\phi_g:M\to M$ is a symplectic map for 
any $g\in G$.
\begin{thm}[\bf{Marsden-Weinstein reduction Theorem}, \cite{MarsWein}]\label{SymplecticReduction}
 Let $\phi:G\times M\to M$ be a free and proper symplectic 
action of a Lie group $G$ on a symplectic manifold $(M,\Omega)$. 
If $J:M\to\g^*$ is an $\Ad^*$-equivariant momentum map 
associated with $\phi$ and $\nu\in\g^*$, then 
$M_\nu=J^{-1}(\nu)/G_\nu$ is a symplectic manifold with 
symplectic $2$-form $\Omega_\nu$ characterized by
\begin{equation*}\label{ReducedSymplecticForm}
 \pi_\nu^*\Omega_\nu=i_\nu^*\Omega,
\end{equation*}
where $\pi_\nu:J^{-1}(\nu)\to M_\nu$ is the canonical projection 
and $i_\nu:J^{-1}(\nu)\hookrightarrow M$ is the canonical 
inclusion.
\end{thm}

In fact, the Poisson structure associated with $\Omega_\nu$ 
is just the reduced Poisson structure obtained by Theorem 
\ref{PoissonReduction} (see \cite{MarsRat}).

Other interesting examples of Poisson manifolds are the  cosymplectic manifolds. For this type of structures,  Albert (\cite{Alb}) obtained 
a cosymplectic reduction Theorem. We recall that a \emph{cosymplectic structure} on a manifold $M^{2n+1}$ of odd dimension $2n+1$ is a couple $(\omega,\eta)$, where 
$\omega$ is a closed $2$-form on $M$ and $\eta$ is a 
closed $1$-form on $M$ such that $\eta\wedge\omega^n$ 
is a volume form. If $(\omega,\eta)$ is a cosymplectic structure on a manifold $M$ then there exists a unique vector field ${\mathcal R}$ on $M$, the Reeb vector field, satisfying the conditions $i_{\mathcal R}\omega=0$ and $i_{\mathcal R}\eta=1$. On the other hand, 
the Hamiltonian vector field $\Ha_\tau$ associated with a 
function $\tau:M\to\Rr$ is characterized by 
\begin{equation}\label{HamiltonianFieldCosymplNew}
 i_{\Ha_\tau}\omega=\d\tau-\Ree(\tau)\eta,\qquad\eta(\Ha_\tau)=0.
\end{equation}
An action $\phi:G\times M\to M$ of a Lie group $G$ on a cosymplectic 
manifold $(M,\omega,\eta)$ is said to be \emph{cosymplectic} 
if $\phi_g:M\to M$ preserves the cosymplectic structure, for any $g\in G$.

\begin{thm}[\bf{Cosymplectic reduction Theorem}, \cite{Alb}]\label{CosymplecticReduction}
Let $\phi:G\times M\to M$ be a free, proper and cosymplectic 
action of a Lie group $G$ on a cosymplectic manifold 
$(M,\omega,\eta)$. 
Suppose that $J:M\to\g^*$ is an $\Ad^*$-equivariant 
momentum map associated with $\phi$ such that 

$\Ree(J_\xi)=0$ for any $\xi\in\g$, 
where $\Ree$ is the Reeb vector field of $M$.

Then, for any $\nu\in\g^*$, $M_\nu=J^{-1}(\nu)/G_\nu$ is a 
cosymplectic manifold with cosymplectic structure 
$(\omega_\nu,\eta_\nu)$ characterized by
\begin{equation}\label{ReducedCosymplecticStruct}
 \pi_\nu^*\omega_\nu=i_\nu^*\omega, \qquad
   \pi_\nu^*\eta_\nu=i_\nu^*\eta,
\end{equation}
where $\pi_\nu:J^{-1}(\nu)\to M_\nu$ is the canonical projection 
and $i_\nu:J^{-1}(\nu)\hookrightarrow M$ is the canonical 
inclusion.

Moreover, the restriction $\Ree_{|J^{-1}(\nu)}$ of $\Ree$ is tangent to $J^{-1}(\nu)$ and 
$\pi_\nu$-projectable on the Reeb vector field $\Ree_\nu$ of $M_\nu$.
\end{thm}

Using \eqref{HamiltonianFieldCosymplNew} it's easy to show 
that, if $\rho$ is a $G$-invariant function on 
$M$ and $\rho_\nu$ is the function on $M_\nu$ such that 
$\rho_\nu\circ\pi_\nu=\rho_{|J^{-1}(\nu)}$, then the restriction to $J^{-1}(\nu)$ 
of $\Ha_\rho$ is tangent to $J^{-1}(\nu)$ and
\begin{equation}\label{HamiltonianRelatedCosympl}
 T_x\pi_\nu(\Ha_\rho(x))=\Ha_{\rho_\nu}(\pi_\nu(x)),
  \qquad\mbox{for all }x\in J^{-1}(\nu).
\end{equation}
This relation is formally the same of 
\eqref{HamiltonianRelatedPoisson}. This fact suggests 
that the Poisson bracket induced by the reduced 
cosymplectic structure is just the reduced Poisson bracket.
In fact, as in the symplectic case, we have the following result
\begin{prop}
 Under the same hypotheses as in Theorem \ref{CosymplecticReduction}, 
the Poisson bracket associated with $(\omega_\nu,\eta_\nu)$ 
is just the reduced Poisson bracket $\pbrEm_\nu$ deduced from Theorem 
\ref{PoissonReduction}.
\end{prop}
\begin{proof}
Denote by $\pbrEm_\nu$ (resp.~$\pbrEm'_\nu$) the Poisson bracket on $M_\nu$ 
obtained from Theorem \ref{PoissonReduction} by reducing the Poisson bracket 
$\pbrEm$ on $M$ (resp.~induced by the reduced cosymplectic structure 
$(\omega_\nu,\eta_\nu)$). Let $\rho_\nu,\tau_\nu\in C^\infty(M_\nu)$ 
and $\rho,\tau$ be arbitrary $G$-invariant extensions of 
$\rho_\nu\circ\pi_\nu, \tau_\nu\circ\pi_\nu$ respectively. Then, for any 
$x\in J^{-1}(\nu)$, using \eqref{ReducedPoissonBracket} and \eqref{HamiltonianRelatedCosympl}
 we have that
\[
  \pbr{\rho_\nu,\tau_\nu}_\nu(\pi_\nu(x))  
   = \pbr{\rho,\tau}(x) = \Ha_\tau(x)(\rho)= \Ha_{\tau_\nu}(\pi_\nu(x))(\rho_\nu)= \pbr{\rho_\nu,\tau_\nu}'_\nu(\pi_\nu(x)).
\]
Since $\pi_\nu$ is surjective, $\pbrEm_\nu=\pbrEm'_\nu$.
\end{proof}

\end{document}